\documentclass[11pt]{amsart}
\usepackage {amsmath, amssymb,   epsfig, a4wide, enumerate, psfrag}
\usepackage[utf8]{inputenc}
\usepackage{xcolor}
\usepackage[english]{babel}
\usepackage{csquotes}
\numberwithin{equation}{section}

\date{\today}

\author{Zhuchao Ji}

\title{ Non-uniform hyperbolicity in polynomial skew products}

\address{Sorbonne Universités, Laboratoire de Probabilités, Statistique et Modélisation (LPSM, UMR 8001),   4 place Jussieu, 75252 Paris Cedex 05, France}
\address{Shanghai Center for Mathematical Sciences, Jiangwan Campus, Fudan University, No 2005 Songhu Road, Shanghai, China 200438}
 \email{zhuchaoji.math@gmail.com}


\DeclareMathOperator{\Mod}{Mod}
\DeclareMathOperator{\diam}{diam}
\DeclareMathOperator{\dist}{dist}
\DeclareMathOperator{\Crit'}{Crit'}
\DeclareMathOperator{\C}{Crit}
\DeclareMathOperator{\meas}{meas}
\DeclareMathOperator{\supp}{supp}

\DeclareMathOperator{\loc}{loc}
\DeclareMathOperator{\Comp}{Comp}
\DeclareMathOperator{\diamComp}{diam\;Comp}

\newtheorem{theorem}{Theorem}[section]
\newtheorem{definition}[theorem]{Definition}
\newtheorem{proposition}[theorem]{Proposition}
\newtheorem{corollary}[theorem]{Corollary}
\newtheorem{lemma}[theorem]{Lemma}
\newtheorem{remark}[theorem]{Remark}

\newtheorem*{theorem*}{Theorem}
\newtheorem*{definition*}{Definition}
\newtheorem*{lemma*}{Lemma}
\newtheorem*{proposition*}{Proposition}

\begin{document}

	\maketitle
	
	\begin{abstract}
	Let $f:\mathbb{C}^2\to \mathbb{C}^2$ be a polynomial skew product which leaves invariant an  attracting vertical line $ L $. Assume moreover $f$ restricted to $L$ is non-uniformly hyperbolic, in the sense that $f$ restricted to $L$ satisfies one of the following conditions: 1. $f|_L$ satisfies Topological Collet-Eckmann and  Weak Regularity conditions. 2. The Lyapunov exponent at every critical value point lying in the Julia set of $f|_{L}$ exist and is positive, and there is no parabolic cycle.  Under one of the above conditions we show that the Fatou set in the basin of $L$ coincides with the union of the basins of attracting cycles, and the Julia set in the basin of $L$ has Lebesgue measure zero. As an easy consequence there are no wandering Fatou components in the basin of $L$.
	\end{abstract}

	\section{Introduction}
	\subsection{Background}
	
	In one-dimensional complex dynamics, i.e. in the theory of dynamics of rational maps on Riemann sphere $\mathbb{P}^1$, the classical Fatou-Julia dichotomy partitions the Riemann sphere into the Fatou set and the Julia set. Let $f$ be a rational map on $\mathbb{P}^1$, the {\em Fatou set} $ F(f) $ is  defined as the largest open subset of $\mathbb{P}^1 $ in which the sequence of iterates $\left(  f^n\right)  _{n\geq 0}$ is normal. Its complement is the {\em Julia set} $ J(f) $. A {\em Fatou component} is a connected component of $ F(f) $. A Fatou component is called {\em wandering} if it is not pre-periodic.
	One can show that the Fatou set is either empty or an open and dense subset. The dynamics on the Fatou set is completely understood, due to the work of Fatou, Julia, Siegel and Herman, supplemented with Sullivan's non-wandering domain theorem \cite{sullivan1985quasiconformal}: the orbit of any point in the Fatou set eventually lands in an attracting basin, a parabolic basin, a Siegel disk or a Herman ring. See Milnor \cite{milnor2011dynamics} for a self-contained proof.
	
	\medskip
	\par If in addition $f$ satisfies some {\em non-uniformly hyperbolic} conditions, the measurable dynamics of $f$  can also be understood. There are various hyperbolic conditions, such as uniform hyperbolicity, sub-hyperbolicity, the {\em Collet-Eckmann} condition ({\em CE} for short), the {\em Topological Collet-Eckmann} condition ({\em TCE} for short), the condition that Lyapunov exponent at  all critical values exist and is positive, and $f$ has no parabolic cycles ({\em Positive Lyapunov} for short), the {\em Weak regularity condition} ({\em WR} for short), etc.  We say that $f$ satisfies TCE  if there is an "Exponential shrinking of components" on the Julia set, see the precise definition in Definition 2.6. 
	\begin{theorem*}{\bf (Przytycki, Rivera-Letelier, Smirnov \cite[Theorem 4.3]{przytycki2003equivalence})}
		Let $f$ be a TCE rational map on $ \mathbb{P}^1 $ with degree at least 2 and such that $J(f)\neq \mathbb{P}^1$. Then the Fatou set $ F(f) $ is equal to the union of a finite number of attracting basins, and the Julia set $ J(f) $ has Hausdorff dimension strictly smaller than 2 (hence it has area zero).
	\end{theorem*}
	
	\medskip
	
	\par In higher dimensional complex dynamics,  one of the major problems is to study the dynamics of {\em holomorphic endomorphisms} of $\mathbb{P}^k$, $k\geq 2$. The Fatou and Julia sets can be defined similarly. Unlike the one-dimensional case, little is known about the dynamics on the Fatou set in higher dimension. It is known that Sullivan's non-wandering domain theorem does not hold in general. Indeed Astorg, Buff, Dujardin, Peters and Raissy \cite{astorg2014two} constructed a holomorphic endomorphism $ h:\mathbb{P}^2\longrightarrow \mathbb{P}^2$ induced by a {\em polynomial skew product}, possessing a wandering Fatou component.
	\medskip
	
	\par A polynomial skew product $ f $ is a polynomial map from  $ \mathbb{C}^2 $ to $ \mathbb{C}^2 $, of the following form:
	\begin{equation*}
	f(t,z)=(g(t),h(t,z)),
	\end{equation*}
	where $ g $ is a one variable polynomial and $ h$ is a two variables polynomial. We assume that $g$ and $h$ have degree at least 2. In the rest of the paper  a polynomial map or a rational map is asked to have degree at least 2.  See Jonsson \cite{jonsson1999dynamics} for a systematic study of such polynomial skew products, see also Dujardin \cite{dujardin2016non}, Astorg and Bianchi \cite{astorg2018bifurcations}, Boc-Thaler, Fornaess and Peters \cite{boc2015fatou} for  related studies. As the definition suggests, the polynomial skew product leaves invariant a foliation by vertical lines, hence one-dimensional tools can be used. Our first purpose is to study the dynamics of a polynomial skew product on its Fatou set.
	\medskip
	
	 We assume $h$ has the expression
	\begin{equation*}
	h(t,z)=\sum_{i+j\leq n}a_{i,j}t^i z^j.
	\end{equation*}
	If in addition we assume the polynomial skew product $f$ satisfies $\deg  \;g= \deg  \;h=n$, and $a_{0,n}\neq 0$, 
	then $f$ extends to $ \mathbb{P}^2 $ holomorphically. In this case the polynomial skew product is called {\em regular}. The regular polynomial skew products form a sub-class of holomorphic endomorphisms on $ \mathbb{P}^2 $.
	\medskip
	
	To investigate the Fatou set of $ f
	$, let $ \pi_1 $ be the projection to the $ t $-coordinate, i.e. \begin{equation*}\pi_1 :\mathbb{C}^2\to \mathbb{C}, \;\;\pi_1(t,z)=t .
	\end{equation*}
	We first notice that $\pi_1(F(f))\subset F(g)$, and passing to some iterate of $f$, we may assume that  the points in $F(g)$ will eventually land into an immediate  basin or a Siegel disk (no Herman rings for polynomials), thus we only need to study the following semi-local case:
	\begin{equation}\label{form}
	f=(g,h): \Delta\times\mathbb{C}\to \Delta\times\mathbb{C},
	\end{equation}
	where $g(0)=0$ which means the line $L:\left\lbrace t=0\right\rbrace $ is invariant and $ \Delta $ is an immediate attracting or a parabolic basin or a Siegel disk of $g$. The map $f$ is called attracting, parabolic or elliptic respectively when $g'(0)$ is attracting, parabolic, elliptic. The examples of wandering domains constructed in \cite{astorg2014two} are parabolic polynomial skew products. At this stage it remains an interesting problem to investigate the existence of wandering domains for attracting  polynomial skew products, one part of our main theorem answer this question in the negative way under the non-uniformly hyperbolic condition.
	\medskip
	
	\par  In the geometrically attracting case, by  Koenigs' Theorem,  (\ref{form}) is locally conjugated to
	\begin{equation}\label{normal}
	f(t,z)=(\lambda t,h(t,z)),
	\end{equation}
	where $\lambda=g'(0)$. Beware that $h$ is no longer a polynomial in $t$.
	\medskip
	\par In the super-attracting case by Böttcher's Theorem, (\ref{form}) is locally conjugated to
	\begin{equation*}
	f(t,z)=(t^m,h(t,z)),\;m\geq 2.
	\end{equation*}
	
	In both attracting or super-attracting cases \begin{equation*}
	h(t,z)=a_0(t)+a_1(t)z+\cdots+a_d(t)z^d
	\end{equation*}
	is a polynomial in $z$ with coefficients $ a_i(t) $ holomorphic in $ t $ in a neighborhood of $ 0 $. We furthermore assume that $ a_d(0)\neq 0 $, which means the degree of $h(t,z)$ in $z$ is constant for $t\in \Delta$. This condition is needed in the proof of the main theorem, and it is satisfied for regular polynomial skew products. In the rest of the paper, an attracting polynomial skew product is assumed to have the normal form (\ref{normal}), and $ \Delta $ denotes a small disk centered at $ 0 $.
	\subsection{Main theorem and outline of the proof}
	In this paper we show that  under the non-uniformly hyperbolic hypothesis, we can exclude the existence of wandering domain, and give a classification of the dynamics on the Fatou set, and show that the Julia set has Lebesgue measure zero.
	\begin{theorem*}{\bf(Main Theorem)}
		Let $f$ be an attracting polynomial skew product, let $ p=f|_{L} $ be the restriction of $f$ on the invariant fiber $L$. Assume that $p$ satisfies one of the following conditions: 1. $p$ satisfies TCE and WR. 2. $p$ satisfies Positive Lyapunov. Then the Fatou set of $f$ coincides with the union of the basins of attracting cycles, and the Julia set of $f$ has  Lebesgue measure zero. 
	\end{theorem*}
We let $ \infty $ be the point at infinity of $L$. Since $\infty$ can be seen as an attracting fixed point, the Fatou set of $f$ is never empty. The definitions of TCE condition, WR condition and Positive Lyapunov condition are given in section 2. The   basins of attracting cycles are clearly non-wandering, as a consequence  there are no wandering domains in the  basin of $L$. In the rest of the paper we shall prove the main theorem for  $f$  geometrically attracting. In the super-attracting case the proof is completely similar, and left to the reader. Note that in that case the non existence wandering Fatou components was established by Lilov \cite{lilov2004fatou} (see also \cite{ji2020non}).
	
	\medskip
	
	The proof of the main theorem is divided into several steps. In section 2 we recall some preliminaries, and introduce some one-dimensional techniques such as the Koebe distortion lemma, Przytycki's lemma and Denker-Przytycki-Urbanski's lemma (DPU lemma for short). The one-dimensional non-uniformly hyperbolic theory is also introduced. Then proof of the main theorem goes as follows. 
	
	\medskip
	
	\noindent
	{\bf  Step 1:} We start with some definitions. 
	Let $ f: \Delta\times\mathbb{C}\to \Delta\times\mathbb{C} $ be an attracting polynomial skew product, the {\em critical set} of $f$ is defined by  $ \C :=\left\lbrace (t,z)\in \Delta\times \mathbb{C}\;|  \;\frac{\partial h}{\partial z}(t,z)=0\right\rbrace  $. We let the radius of $\Delta$ be small enough so that each connected component of Crit intersect $L$ at exactly one point.  We define \begin{equation*}
	\Crit':=\left\lbrace \text{the union of the connected components of Crit that intersect}\; J(p) \right\rbrace.  
	\end{equation*} 
	
	\begin{definition}
		Let $ x\in \Delta\times \mathbb{C} $ be a point in the immediate basin of $L$, we say that $x$ slowly approach Crit'  if for every $ \alpha >0$, $ \dist _v(f^n(x),\Crit' ) \geq e^{-\alpha n} $ for every sufficiently large $n$.
	\end{definition}
	Here $ \dist _v $ denote the vertical distance which means that $\dist _v(x,y)=|\pi_2(x)-\pi_2(y)|$, where $ \pi_2 $ is the projection to the $ z $-coordinate, i.e. \begin{equation*}\pi_2 :\mathbb{C}^2\to \mathbb{C}, \;\;\pi_2(t,z)=z.
	\end{equation*}
		Let also $ D_v(x,r) $ denote the vertical disk, $ D_v(x,r)=\left\lbrace y \in\Delta\times \mathbb{C}:\pi_1(y)=\pi_1(x),\;\dist _v(x,y)<r \right\rbrace $. 
	\medskip
	\par This notion of slow approach was introduced by Levin, Przytycki and Shen in one-dimensional complex dynamics in \cite{levin2016lyapunov}. Next we show that
	\begin{theorem}
		Lebesgue a.e. $ x\in \Delta\times \mathbb{C} $ slowly approach Crit'.
	\end{theorem}
	This is proved in sections 3 and 4. In section 3 the existence of a stable manifold at each critical value in $ J(p) $ and the properties of renormalization maps associated with a critical value variety are studied. In section 4 we use the techniques developed in section 3 to prove Theorem 1.2.  Step 1 is where we need the TCE and WR conditions or the Positive Lyapunov condition, in the remaining steps the TCE condition alone is sufficient for the proof.
	
	\medskip
	
	\noindent
	{\bf  Step 2:} We define vertical Lyapunov exponent at one point as follows:
	\begin{definition}
		Let $ x\in \Delta\times \mathbb{C} $ be a point in the immediate basin of $L$, the {\em lower vertical Lyapunov exponent} is defined by
		\begin{equation*}
		\chi_{-}(x):=\liminf _{n\to\infty}\frac{1}{n} \log|Df^n|_{x}(v)|. 
		\end{equation*}
	\end{definition}
	Where $v=(0,1)$ is the unit vertical tangent vector.
	\medskip
	
	It is well-known that the one-dimensional attracting basins of $p$ extend to two-dimensional attracting basins, for example see \cite{lilov2004fatou} or \cite[Section 3]{ji2020non}. These two-dimensional attracting basins correspond to non-wandering Fatou components. We let $ W^s(J(p)) $ denote the {\em stable set} of $ J(p) $,
	\begin{equation*}
	W^s(J(p)):=\left\lbrace x\in \Delta\times \mathbb{C}:\lim_{n\to\infty}\dist (f^n(x),J(p))= 0\right\rbrace .
	\end{equation*} It is easy to see that assuming $p$ satisfies TCE, $ W^s(J(p)) $ is the union of the wandering Fatou components together with the Julia set $ J(f) $. We show that
	\begin{theorem}
		If $ x\in W^s(J(p)) $ slowly approach Crit', then $ \chi_{-}(x)\geq \log \mu_{Exp} $, where $ \mu_{Exp}>1 $ is the constant appearing in the definition of the TCE condition.
	\end{theorem}
	This is proved in section 5. Then by Theorem 1.2 Lebesgue a.e. point $x\in W^s(J(p))$ satisfies $ \chi_{-}(x)\geq \log \mu_{Exp} $. This already implies the non-existence of wandering Fatou component thanks to the fact that points in Fatou set can not have positive Lyapunov exponent. Thus  $ W^s(J(p)) $  coincide with the Julia set $ J(f) $.
	
	\medskip
	
	\noindent
	{\bf  Step 3:} Finally by using an adaption of a so called \enquote{telescope argument} in \cite[Theorem 1.5]{levin2016lyapunov}, we show that 
	\begin{theorem}
		The Julia set $ J(f) $ has Lebesgue measure zero.
	\end{theorem}
	This is proved in section 6, and the proof is complete.
	\medskip
\par In Appendix A we  study the relations between TCE condition, CE condition,  WR condition and Positive Lyapunov condition. In Appendix B we exhibit some families of polynomial maps satisfying the conditions in the main theorem.
	
	\subsection{Previous results}
	The first result of non-wandering domain theorem for polynomial skew products goes back to Lilov \cite{lilov2004fatou}. In his PhD thesis Lilov proved that super-attracting polynomial skew products do not have wandering Fatou components. He actually showed  a stronger result, namely that there can not have vertical wandering Fatou disks. 
	
	\medskip
    \par In the geometrically attracting case, there are many works trying to understand the dynamics in the atrracting basin of the invariant line. Peters and Vivas showed in \cite{peters2016polynomial} that there is an attracting  polynomial skew product with a wandering vertical Fatou disk. This result does not answer the existence question of wandering Fatou components, but showed that the question is considerably more complicated than in the super-attracting case. On the other hand, by using a different strategy from Lilov's, Peters and Smit in \cite{peters2018fatou} showed that the non-wandering domain theorem holds in the attracting case, under the assumption that the dynamics on the invariant fiber is sub-hyperbolic. The author showed that the non-wandering domain theorem holds in the attracting case, under the assumption that the multiplier is sufficiently small, following Lilov's strategy, see \cite{ji2020non}. 
    
    \medskip
    \par In the parabolic case, the examples of wandering domains are constructed in \cite{astorg2014two}, as we have mentioned. See also the recent paper \cite{astorg2019wandering} and \cite{hahn2021polynomial}.
    
    \medskip
    \par The elliptic case was studied by Peters and Raissy in \cite{peters2017elliptic}. See Raissy \cite{raissy2017polynomial} for a survey of the history of the investigation of wandering domains for polynomial skew products. 
    
    \medskip
    \par To the best of our knowledge, Theorem 1.5 is the first time where the zero measure of Julia set is shown for non-hyperbolic $p$ (it is in general not true when no conditions of $p$ are assumed, as even in one dimension Julia set can has positive Lebesgue measure, cf. \cite{buff2012quadratic} and \cite{avila2015lebesgue}). The previous results we have mentioned only consider the dynamics on the Fatou set. However when $p$ is uniformly hyperbolic, it is well-known that the Julia set has zero measure and the Fatou set coincide with the union of the basins of attracting cycles. In fact the  stable set of $W^s(J(p))$ is foliated by stable manifolds, and this foliation is absolutely continuous hence $W^s(J(p))$ has Lebesgue measure zero.
	
	\medskip
	\par {\em Acknowledgements.} I would like to thank my adviser Romain Dujardin for his advice, help and encouragement during the course of this work. I also would like to thank Jacek Graczyk for useful discussion, Theorem B.1 is due to him.
	\medskip
	
	\section{Preliminaries}
	
	In this section we introduce some one-dimensional tools used in the proof of the main theorem. Let $f:\;\mathbb{P}^1\to \mathbb{P}^1$ be a rational map , let Crit denote the set of critical points, and let Crit' denote the set of critical points lie in the Julia set. Let $CV(f)$ be the critical value set. We fix a Riemannian metric on $ \mathbb{P}^1 $, and $D(x,\varepsilon)$ denotes a small disk centered at $x$ with radius $\varepsilon$.

	\subsection{Some technical lemmas}
	In \cite[Lemma 1]{przytycki1993lyapunov}, Przytycki introduced a fundamental lemma which concerns the recurrence properties of small neighborhood of Crit'. 
	
	\begin{lemma}[\bf Przytycki]\label{Przytycki}
		Let $ c\in \Crit'$. There exist a constant $ C>0 $ such that for every $ \varepsilon>0 $ and $n>0$, if $ f^n\left(D(c,\varepsilon)\right) \cap D(c,\varepsilon)\neq \emptyset $, then $n\geq C\log\frac{1}{\varepsilon}$.
	\end{lemma}
	
	We define a positive valued function on $ \mathbb{P}^1 $ as follows
	
	\begin{definition}
		Let $ c\in Crit' $, define a positive valued function $ {\phi_c(x)} $ by 
		\begin{equation*}
		\phi_c(x):=\left\{
		\begin{aligned}
		-&\log  \dist \;(x,c), & \; \text{if}\;x\neq c  \\
		&\infty, &\;\text{if}\; x=c. \\
		\end{aligned}
		\right.
		\end{equation*}
		
	\end{definition}
	In terms of using the function $ {\phi_c(x)} $, the Lemma \ref{Przytycki} can be reformulated as: there exists a constant $Q>0$, such that for every $ x\in\mathbb{P}^1 $, $c\in \Crit' $, $n\geq 1$, we have
	\begin{equation*}
	\min \left( \phi_c(x),\; \phi_c(f^n(x))\right) \leq Qn .
	\end{equation*}

\medskip
	In a later paper by Denker, Przytycki, Urbanski \cite[Lemma 2.3]{denker1996transfer}, Lemma \ref{Przytycki} was generalized as the following DPU lemma. We let $\phi(x):=\max_{c\in \Crit' } \phi_c(x)$.
	
	\begin{lemma}[\bf Denker-Przytycki-Urbanski]\label{DPU}
		There exist a constant $ Q>0 $ such that for all $n\geq 0$ we have
		\begin{equation*}
			\sum_{\begin{subarray}{c}j=0\\\text{except}\;M\;\text{terms}\end{subarray}}^{n-1}\phi(f^j(x))\leq Qn,
		\end{equation*}
		where the summation over all but at most $M=\#$Crit' indices.
	\end{lemma}

	Note that the original statement differs slightly. This formulation also appeared in \cite[ Lemma 2.5]{ji2020non}. Lemma \ref{Przytycki} will be used in section 4 and Lemma \ref{DPU} will be used in section 6.
	\medskip
	
	Next we introduce a version of  the Koebe distortion lemma for multivalent maps. We refer to \cite[Lemma 1.4]{przytycki1998iterations} and \cite[Lemma 2.1]{przytycki1998porosity} for more details. Consider a disk $ D(x,\delta) $ of radius $ \delta
	$ centered at $x$, let $ W $ be a connected component of $f^{-n}(D(x,\delta))$, assume that $ f^n $ restricted to $ W $ is $ D $-critical, that is $f^n$ has at most $D$ critical points counted with multiplicity. Then $ f^n|_W $ has distortion properties similar to univalent maps. In the following we assume $ \delta $ smaller than $\diam  \mathbb{P}^1/2$.
	
	\begin{lemma}\label{Koebe}
		For each $\varepsilon>0$ and $D<\infty$ there are constants $C_1(\varepsilon, D)>0$ and $C_2(\varepsilon, D)>0$ such that the following holds.
		\medskip
		\par Let $D(x,\delta)$ denotes the ball in $\mathbb{P}^1$ centered at $x$ with radius $\delta$.  Assume that $ W $  is a simply connected domain in $\mathbb{P}^1$ and $F: W\to D(x,\delta) $   is a proper holomorphic map. Let $W'\subset W$ be a connected component of $F^{-1}(D(x,\delta/2))$. Assume further that $\mathbb{P}^1\setminus W$ contains a disk of radius $ \varepsilon $ and that $ F $ is $ D $-critical on $ W $. Then for every $y\in W'$,
		\begin{equation}
		|F'(y)|\diam \;(W')\leq C_1\delta.
		\end{equation}
		In addition $W'$  contains a disk $B$ of radius $r$ around every pre-image of $F^{-1}(x)$  contained in $W'$, with
		\begin{equation}
		r\geq C_2\diam (W').
		\end{equation}

		\par Assume further that $W''$ is a connected component of $ F^{-1}(B') $, where $B'\subset D(x,\delta/2)$ is a disk, then there exist a constant $C_3(\varepsilon, D)>0$ such that 
		\begin{equation}
			\frac{\diam \;W''}{\diam \;W'}\leq C_3\left(\frac{\diam (B')}{\delta}\right)^{2^{-D}} .
		\end{equation}
	Finally if  $R$ is a measurable subset of $D(x,\delta/2)$, there exist a constant $ C_4(\varepsilon, D)>0 $ such that 
		\begin{equation}
		\frac{\meas \left( F^{-1}(R)\cap W'\right) }{\meas \;W'}\leq C_4\left(\frac{\meas (R)}{\delta^2} \right)^{2^{-D}}.
		\end{equation}
		where $\meas$ denotes the Lebesgue measure induced by the Riemannian metric on $ \mathbb{P}^1 $.
	\end{lemma}
Note that we will use this lemma only for $F$ being a polynomial, so the assumption that $W$ is simply connected and $\mathbb{P}^1\setminus W$ contains a disk are automatically satisfied.
\begin{proof}
	The inequalities (2.1) and (2.2) were proved in \cite[Lemma 2.1]{przytycki1998porosity} and the inequality (2.3) was proved in \cite[Lemma 1.4]{przytycki1998iterations}. Here we prove inequality (2.4).
	\medskip
	\par  By the Riemann mapping theorem there is a surjective univalent map $\psi:D(0,1)\to W$. We consider the composition $F\circ\psi:D(0,1)\to D(x,\delta)$. By the classical Koebe distortion lemma for univalent maps, (2.4) is true for $D=0$. To prove (2.4), it is sufficient to prove the following: Let $G:D(0,1)\to D(0,1)$ be a degree $D$ Blaschke product on the unit disk,  there exist a constant $ C_4(D)>0 $ such that if $R$ is a measurable set of $D(0,1/2)$, then
	\begin{equation}
\meas  G^{-1}(R) \leq C_4\meas (R)^{2^{-D}}.
	\end{equation}

	\par We let  $a_i\in D(0,1)$ be the critical points of $G$,  $1\leq i\leq n$. Thus we have $n\leq D$. We denote $A= \meas (R)$. It is sufficient to prove (2.5) for $A$ small. For small $\varepsilon$ we cover $R$ by  $\varepsilon$-disks such that $N\pi\varepsilon^2\leq 2A$, where $N$ is the number of disks in the covering. For $\delta>0$ small there exist a uniform constant $M$ such that $\dist (G(x),G(a_i))\geq \delta$ for every $i$ and $G(x)\in D(0,3/4)$  imply $|G'(x)|\geq M \delta^{1/2}$. If an $\varepsilon$-disk $ D_\varepsilon$   is disjoint from the union $\bigcup_{i=1}^n D(G(a_i),A^{1/2})$, then by the change of variable formula we have
	\begin{equation*}
	\meas G^{-1}(D_\varepsilon ) \min_{x\in G^{-1}(D_\varepsilon ) } |G'(x)|^2\leq D\pi \varepsilon^2.
	\end{equation*}
From $\min_{x\in G^{-1}(D_\varepsilon ) } |G'(x)|^2\geq M^2A^{1/2}$ we get
\begin{equation*}
\meas G^{-1}(D_\varepsilon )\leq \frac{D\pi \varepsilon^2}{M^2A^{1/2}}.
\end{equation*}
Let $B_1$ be the union of all $\varepsilon$-disks  disjoint from $\bigcup_{i=1}^n D(G(a_i),A^{1/2})$. Then we have
\begin{equation*}
\meas G^{-1}(B_1)\leq \sum_{D_\varepsilon \cap\left(  \bigcup_{i=1}^n D(G(a_i),A^{1/2})\right) = \emptyset}\meas G^{-1}(D_\varepsilon )\leq \frac{ND\pi \varepsilon^2}{M^2A^{1/2}}\leq \frac{2DA^{1/2}}{M^2}.
\end{equation*}
Let $B_{2,i}$ be the union of all $\varepsilon$-disks not disjoint from $D(G(a_i),A^{1/2})$, for $1\leq i\leq n$. Let $B_2=\bigcup_{i=1}^n B_{2,i}$. By (2.3) for small $\varepsilon$ ($\varepsilon\leq A^{1/2}$, say), there is a constant $C_3(D)>0$ such that 
\begin{equation*}
\meas G^{-1}(D(G(a_i),A^{1/2}+\varepsilon))\leq C_3 A^{2^{-D}}.
\end{equation*} Thus we have
\begin{equation*}
\meas G^{-1}(B_2)\leq \sum_{i=1}^n\meas G^{-1}(B_{2,i})\leq \sum_{i=1}^n\meas G^{-1}(D(G(a_i),A^{1/2}+\varepsilon))\leq D C_3 A^{2^{-D}}.
\end{equation*}
The last inequality holds since $n\leq D$. Finally we have
\begin{equation*}
\meas G^{-1}(R)\leq \meas G^{-1}(B_1)+\meas G^{-1}(B_2)\leq \frac{2DA^{1/2}}{M^2}+D C_3 A^{2^{-D}}.
\end{equation*}
Setting $C_4:=DC_3+2D/M^2$ the proof is complete.
\end{proof}
\medskip
	Lemma \ref{Koebe} will be used frequently in the rest of the paper.
	
	\subsection{One-dimensional non-uniformly hyperbolic theory}
	
	\par A rational map $f$ is  uniformly hyperbolic if $f$ expands a Riemannian metric on a neighborhood of $J(f)$. This is equivalent to Smale's Axiom A, and is equivalent to the condition that the closure of the post critical set $ \overline{PC(f)} $ is disjoint from $J(f)$. The measurable dynamics of $f$ is well-understood: the Fatou set is the union of finitely attracting basins, the Hausdorff dimension of $J(f)$ is equal to the Minkowski dimension of $ J(f)$ and is smaller than $2$, and there is a unique invariant probability measure $\mu$ such that $\supp (\mu)=J(f)$ which is absolutely continuous with respect to the $ \delta$-dimensional Hausdorff measure ($\delta$ is the Hausdorff dimension of $J(f)$). It can be shown that $\mu$ is mixing (hence ergodic) and has positive entropy. It is widely conjectured that uniformly hyperbolic maps are dense in the parameter space of fixed degree. This is known as Fatou conjecture and is a central problem in one-dimensional complex dynamics. Many weaker notions such as sub-hyperbolicity, semi-hyperbolicity have been defined. See \cite[Section 19]{milnor2011dynamics}, \cite{carleson1994julia} for more details. 
	
	\medskip
	\par  Non-uniformly hyperbolic theory, also known as Pesin theory, is a generalization of uniformly hyperbolic theory. In Pesin theory we only require an invariant hyperbolic measure rather than the presence of invariant expanding and contracting directions. In this subsection we introduce some strong notions of non-uniform hyperbolicity in one-dimensional complex dynamics.
	
	\begin{definition}\label{CE}
		A rational map $f$ satisfies CE if there exists $ \mu_{CE}>1 $ and $C>0$ such that for every point $ c\in\Crit'  $ whose forward orbit does not meet other critical points, and every $n\geq 0$ we have
		\begin{equation*}
		|(f^n)'(f(c))|\geq C\mu_{CE}^n.
		\end{equation*}
	In addition we ask that there are no parabolic cycles.
	\end{definition}
	The CE condition was first introduced by Collet and Eckmann in \cite{collet1983positive} for $ S $-unimodal maps of an interval. The CE condition was introduced in complex dynamics by Przytycki in \cite{przytycki1998iterations}. The TCE condition was first introduced by Przytycki and Rohde in \cite{przytycki1998porosity}, as a generalization of the CE condition.
	
	\begin{definition}
		A rational map $f$ satisfies TCE if there exist $\mu_{Exp}>1$ and $r>0$ such that for every $ x\in J(f) $, every $n\geq 0$ and every connected component $W$ of $ f^{-n}(D(x,r)) $ we have that
		\begin{equation*}
		\diam \;(W)\leq \mu_{Exp}^{-n}.
		\end{equation*}
		
	\end{definition}
	
	There are various equivalent characterization of the TCE condition, see \cite{przytycki2003equivalence}. The following inclusions are strict:
	\begin{equation*}
	\text{uniform hyperbolicity}\subsetneqq \text{sub-hyperbolicity}\subsetneqq\text{CE}\subsetneqq\text{TCE}.
	\end{equation*}

	\par It was proved by Aspenberg \cite{aspenberg2013collet} that the set of non-hyperbolic CE maps has positive measure in the parameter space of rational maps of fixed degree, see also Rees \cite{rees1986positive}. In the family of unicritical polynomials, it was shown by Graczyk-Swiatek \cite{graczyk2000harmonic} and Smirnov \cite{smirnov2000symbolic}  that for a.e. $ c\in\partial \mathcal{M}_d $ in the sense of harmonic measure $f_c=z^d+c$ satisfies the CE condition, where $ \mathcal{M}_d $ is the connectedness locus, and $\partial \mathcal{M}_d$ is the bifurcation locus.  We list some useful property of TCE maps.
	
	\begin{proposition}\label{TCE}
		Let $f$ be a TCE map such that $J(f)\neq \mathbb{P}^1$ then\\
		\par (1) The Fatou set $F(f)$ is the union of attracting basins.\\
		\par (2) The Hausdorff dimension $\delta$ of $J(f)$ is equal to the Minkowski dimension of $ J(f)$ and is smaller than 2.\\
		\par (3) There is a unique invariant probability measure $\mu$ such that $\supp (\mu)\subset J(f)$ and $\mu$ is absolutely continuous with respect to the conformal measure with exponent $\delta$. Moreover  $\mu$ is exponentially mixing (hence ergodic) and has positive Lyapunov exponent.
	\end{proposition}
	For the proof see \cite{przytycki2003equivalence} and \cite{przytycki2007statistical}. For more about measurable dynamics on $J(f)$, we refer the reader to Przytycki and Rivera-Letelier \cite{przytycki2007statistical}, Graczyk and Smirnov \cite{graczyk2009non}, and Rivera-Letelier and Shen \cite{rivera2014statistical}.
	
	\medskip
	\par In our presentation of the main theorem, we also ask that $p$ satisfies WR or Positive Lyapunov. These additional conditions are used to construct stable manifold at $ v\in CV(p)\cap J(p) $ in section 3.
	
	\begin{definition}
	A rational map  $f$  satisfies $\text{WR}(\eta,\iota) $ if there exists $ \eta,\iota>0 $ and $C_0>0$ such that for all $ v\in CV(f) $ whose forward orbit does not meet any critical point and for every integer $n\geq 0$, it holds
	\begin{equation*}
	\sum_{\begin{subarray}{c}j=0\\d(f^j(v),\Crit' )\leq \eta\end{subarray}}^{n-1}-\log|f'(f^j(v))|<n\iota+C_0.
	\end{equation*}
		
	\end{definition}
\medskip
	 This condition means that for every $ v\in CV(f)\cap J(f) $: the orbit of $v$ does not come either too close nor too often to Crit'. The following WR condition is stronger than WR($ \eta,\iota $) with fixed $\eta$ and $\iota$.
	  \begin{definition}
	  	A  rational map $f$ satisfies WR if for all $ v\in CV(f) $ whose forward orbit does not meet critical points we have
	  	\begin{equation*}
	  	\lim_{\eta\to 0}\limsup_{n\to\infty} \frac{1}{n}\sum_{\begin{subarray}{c}j=0\\d(f^j(v),\Crit' )\leq \eta\end{subarray}}^{n-1}-\log|f'(f^j(v))|=0.
	  	\end{equation*}
	  \end{definition}
  \medskip
  \par The condition Positive Lyapunov is stronger than CE.
 	
 \begin{definition}
 	A rational map $f$ satisfies Positive Lyapunov if  for every point $ c\in\Crit'  $ whose forward orbit does not meet other critical points the following limit exists and is positive
 	\begin{equation*}
 	\lim_{n\to\infty}\frac{1}{n}\log|(f^n)'(f(c))|>0.
 	\end{equation*}
 	In addition we ask that there are no parabolic cycles.
 \end{definition}
	 \medskip
	 \par The condition WR, as well as the condition  Positive Lyapunov is satisfied for many polynomial maps, see Appendix B.
	 \medskip

\section{Stable manifolds and renormalization maps}
In this section we work under  the assumption that  $p$ satisfies TCE and WR($\eta,\iota$) with small $ \iota $ ($ \iota$ will be determined later) or $p$ satisfies Positive Lyapunov.
We construct stable manifolds at each $ v\in CV(p)\cap J(p) $ under the above assumptions, we also study the renormalization maps associated to each critical value variety.

\subsection{Stable manifold of a critical value}
\par A (local) stable manifold $ W^s_{\loc }(v) $ of $f$ is an embedded complex disk of $ \mathbb{C}^2 $ passing through $v$ such that there exist $\delta>0$ that for every $ x\in  W^s_{\loc }(v) $ and $ n\geq 0 $, $ \dist (f^n(x), f^n(v) )$ decreases exponentially fast. The construction of stable manifold is classical for hyperbolic periodic points and for uniformly hyperbolic invariant sets, cf. \cite{katok1995introduction}. In Pesin's theory we can also construct stable manifold for a.e. point with respect to a hyperbolic invariant probability measure, cf. \cite{barreira2002lyapunov}. Since we deal with a single non-uniformly hyperbolic orbit, we construct a stable manifold at $v$ by using Hubbard and Oberste-Vorth's graph transform associated to a sequence of crossed mappings, cf. \cite{hubbard1994henon}.  In the first three subsections we prove the following theorem:
\begin{theorem}\label{manifold}
Let $f$ be an attracting polynomial skew product and let $p=f|_L$ be the restriction of $f$ to the invariant fiber. Assume $p$ satisfies either TCE and WR, or $p$ satisfies Positive Lyapunov. 
Then every $v\in J(p)\cap CV(p)$ admits a local stable manifold transverse to the invariant fiber $L$.
\end{theorem}
\medskip
We begin with some definitions.
\medskip
\par Let $B_1=U_1\times V_1$, $B_2=U_2\times V_2$ be two bi-disks. Let $\Omega$ be a neighborhood of $\overline{B_1}$, let $f:\Omega\to \mathbb{C}^2 $ be a holomorphic map such that $f(B_1)\cap B_2\neq \emptyset$. Let $\pi_1$ (resp. $\pi_2 $) be the projection map to the first (resp. second) coordinate.
\begin{definition}[Hubbard and Oberste-Vorth]\label{crossed mapping}
The map $f$ is called a {\em crossed mapping} of degree $d$ from $ B_1 $ to $ B_2 $ if there exists $ W_1\subset  U_1\times V_1' $, where $ V_1'\subset V_1 $ is a relatively compact open subset and $ W_2\subset U_2'\times V_2 $, where $ U_2'\subset U_2 $ is a relatively compact open subset, such that
 $ f:W_1\to W_2 $ is bi-holomorphic, and for every $ x\in U_1 $, the mapping
\begin{equation*}
\pi_2\circ f|_{W_1\cap((\left\lbrace x\right\rbrace\times V_1)}:W_1\cap(\left\lbrace x\right\rbrace\times V_1 )\to V_2
\end{equation*}
is proper of degree $d$, and for every $y\in V_2$ the mapping
\begin{equation*}
\pi_1\circ f^{-1}|_{W_2\cap( U_2\times \left\lbrace y\right\rbrace )}:W_2\cap(U_2\times \left\lbrace y\right\rbrace)\to U_1
\end{equation*}
is proper of degree $d$.
\end{definition}
Beware that coordinate are switched as compared to \cite{hubbard1994henon}: here the horizontal direction is contracted and the vertical direction is expanded.
\medskip
\par Let $B$ be the bidisk $D(0,1)\times D(0,1)$. We define the horizontal boundary  as $ \partial_h(B):=\left\lbrace |x|<1,  |y|=1\right\rbrace  $. The vertical boundary $\partial_v(B)$ can be defined similarly.
\medskip
\par 
\begin{definition}
An analytic curve $X$ is called {\em horizontal (resp. vertical) } in $B_1$  if $X$ is defined in a neighborhood of $B_1$, $X\cap B_1\neq\emptyset$ and  $X\cap \overline{\partial_h (B_1)}=\emptyset$ (resp. $X\cap \overline{\partial_v (B_1)}=\emptyset$).
\end{definition}
 It follows that  $\pi_1: X\to U_1$ (resp. $\pi_2: X\to V_1$) is proper of degree $d$, for some integer $d>0$. We call this $d$ the degree of the analytic curve.
\begin{proposition}
If $f:B_1\to B_2$ is a degree 1 crossed mapping and $X$ is a degree $d$ vertical curve in $B_1$, then $\pi_2\circ f:X\cap W_1\to V_2$ is proper of degree $d$, or equivalent to say, $f(X)$ is a  degree $d$ vertical curve in $B_2$ (defined in a neighborhood of $B_2$).
\end{proposition}

For the proof see \cite[Proposition 3.4]{hubbard1994henon}.

\begin{definition}[Dujardin]
$ f $ is called {\em H\'enon-like} from $ B_1 $ to $B_2$ if the following three conditions are satisfied\\
\par (1) $f$ restricted to $ \Omega $ is injective, where $\Omega$ is a neighborhood of $\overline{B_1}$.\\
\par (2) $ f(\partial_h B_1)\cap\overline{B_2}=\emptyset $,\\
\par (3) $ f(\overline{B_1})\cap \partial B_2\subset \partial_h B_2. $\\
\end{definition}
Again, note that horizontal and vertical directions are switched as compared to \cite{dujardin2004henon}.
\begin{proposition}
If $f$ is H\'enon-like from $ B_1 $ to $ B_2 $, then $f$ is a crossed mapping.
\end{proposition}
For the proof see \cite[Proposition 2.3]{dujardin2004henon}.
\medskip
\par The following theorem summarizes our approach to construct stable manifolds.

\begin{theorem}\label{s-manifold}
Let $ B_0=U_0\times V_0 $, $ B_1=U_1\times V_1 $,... be an infinite sequence of bi-disks, and $ f_i:B_i\to B_{i+1}  $ be of degree 1 crossed mapping, with $ V_i' $ simply connected (with notation as in Definition \ref{crossed mapping}) such that the modulus $\Mod(V_i\setminus V_i')$ is uniformly bounded from below. Then the set 

\begin{equation*}
W^s_{\left(  f_n\right)}=\left\lbrace (x,y)\in B_0 \;|\;f_n\circ\cdots{\circ f_0\left( x,y\right) }\in B_n \;\;\text{for all} \;n\geq 0\right\rbrace 
\end{equation*}
is a degree 1 horizontal curve in $ B_0 $. $ W^s_{\left(  f_n \right) } $ is called the stable manifold for the sequence of crossed mappings.
\end{theorem}
For the proof see \cite[Corollary 3.12]{hubbard1994henon}.
\medskip
\par We also need some one-dimensional preparations.  The following lemma is due to Przytycki and Rivera-Letelier  (\cite[Lemma 3.3]{przytycki2007statistical}).
\begin{theorem}\label{BC}
	Let $f$ be a rational function satisfying TCE with constants $\mu_{Exp}>1$ and $r_0>0$. Then the following assertions hold.
	\medskip
\par  There are constants $C_0> 0$ and $\theta_0\in (0,1)$ such that for every $r\in (0,r_0)$, every integer $n\geq 1$, every $x\in J(f)$ and every connected component $W$ of
$f^{-n}(B(x, r))$, we have
\begin{equation*}
 \text{diam}\;(W) \leq  C_0 \mu_{Exp}^{-n}r^{\theta_0}.
 \end{equation*}
\end{theorem}
\medskip
\par We now use the above property to show that  $p$ is {\em hyperbolic away from Crit'} in the following sense.
 \begin{lemma}\label{hyperbolic away}
 	Assume $p$ satisfies TCE. Let $\eta>0$ be a constant. Let $\left\lbrace x_0, x_1, \dots, x_{N-1}\right\rbrace \subset J(p)$ be a segment of orbit such that for every $0\leq i\leq N-1$, $ \dist(x_i,\Crit' )>\eta $, then there exist $C_1$ and $\alpha$  uniform constants such that $|(p^{N})'(x_0)|\geq C_1\eta^{\alpha}\mu_{Exp}^{N}$. 
 \end{lemma}
\begin{proof}
We use the same notations $C_0$ and $\theta_0$ as in Theorem \ref{BC}. It is sufficient to prove the lemma for $\eta$ small. Let $r$ be a constant satisfying $\eta=C_0 r^{\theta_0}$ such that $r<r_0$. Let $W_i$ be the pull back of $D(x_N,r)$ by $ p^{i} $ at $x_{N-i}$ for $0\leq i\leq N$. We show that $W_i\cap \Crit'=\emptyset$ for every $i$. Assume that $W_i\cap \Crit'\neq \emptyset$ for some $i$. By Theorem \ref{BC} we have $\diam(W_i)\leq C_0\mu_{Exp}^{-i}r^{\theta_0}\leq \eta$. On the other hand since $ \dist(x_i,\Crit' )>\eta $ for every $i$,  we get $\diam(W_i)>\eta$, which is a contradiction. Thus $W_i\cap \Crit'=\emptyset$ for every $i$ and $ p^N $ restricted to $ W_N $ is univalent.
\medskip
\par  By TCE we have $\diam  W_N\leq \mu_{Exp}^{-N}$, by the classical Koebe distortion theorem we have
\begin{equation*}
|(p^{N})'(x_0)| \;\diam W_N \geq r/4=1/4 C_0^{-1/\theta_0} \eta^{1/\theta_0}.
\end{equation*}
Taking $C_1=1/4 C_0^{-1/\theta_0}$ and $\alpha=1/\theta_0$ we get that $|(p^{N})'(x_0)|\geq C_1\eta^{\alpha}\mu_{Exp}^{N}$. 
\end{proof}
\medskip
\subsection{The TCE+WR case}
   \par In this subsection, for an attracting skew product $f$ such that $p$ satisfies TCE and WR($\eta,\iota$) with small $\iota$, we construct a sequence of bi-disks $ \left\lbrace B_i\right\rbrace $ with low exponential size such that for every integer $i$, $f$ is a degree 1 crossed mapping from $ B_i $ to $ B_{i+1} $ with $B_i$ centered at $p^i(v)$. We use the same  notations  as in Lemma \ref{hyperbolic away}. We fix a constant $0<\varepsilon_0\ll\min\left\lbrace |\lambda|^{-1/3}-1,\;\mu_{Exp}-1\right\rbrace $. In the following we choose an integer $N$ such that $ C_1\eta^{\alpha}\mu_{Exp}^{N}\geq (1+\varepsilon_0)^N $, where $\eta$ is given by WR($\eta,\iota$), $ \alpha $ and $ C_1 $ are as in Lemma \ref{hyperbolic away}. We subdivide the integers into blocks of the form $\left[ iN, (i+1)N\right)$. We say that a block of this subdivision is of {\em first type} if 
 \begin{equation*}
 \prod_{j=iN}^{(i+1)N-1}|p'(p^j(v))|\geq (1+\varepsilon_0)^{N},
 \end{equation*} and we call this subdivision is of {\em second type} if the above inequality does not hold. By Lemma \ref{hyperbolic away} if $\dist(p^j(v),\Crit' )>\eta$ for $iN\leq j<(i+1)N$, then $\left[ iN, (i+1)N\right)$ is of first type. 
\medskip
 \par Let $m\in [iN, (i+1)N)$ be a positive integer, if $m$ is in a block of first type, we define 
 \begin{equation*}
 \mu_m:=(1+\varepsilon_0)\left( \prod_{j=iN}^{(i+1)N-1}|p'(p^j(v))|\right) ^{1/N}.
 \end{equation*} 
 When $m$ is in a block of second type, we define 
 \begin{equation*}
 \mu_m:=(1+\varepsilon_0)^2.
 \end{equation*}
 Note that in both cases we have $\mu_m\geq (1+\varepsilon_0)^2 \geq 1+\varepsilon_0$.
 \medskip
 \par We define
 \begin{equation*}
 r_n:=r_0\prod_{m=0}^{n-1}\frac{a_m}{\mu_m},
 \end{equation*} where $r_0>0$ is a constant to be determined and  $ a_m:=|p'(p^m(v))| $.
 \begin{lemma}\label{size-disk}
 There are constants $C_2,C_3>0$ such that the following estimates of $r_n$ hold for $n\geq 0$: 
 \medskip

   (1) $r_n\leq C_2r_0(1+\varepsilon_0)^{-n}$.
 \medskip  
   
   (2) $ r_n\geq C_3r_0 e^{-\left( \alpha+2 \right)n\iota}(1+\varepsilon_0)^{-2n}, $ where $\alpha$ is as in Lemma \ref{hyperbolic away}.

 \end{lemma}
\begin{proof}
	To prove the first inequality, notice that for every $i\geq 0$, by the definition of $\mu_m$ we have
	 \begin{equation}\label{3.1}
	\prod_{j=iN}^{(i+1)N-1}\mu_j\geq (1+\varepsilon_0)^N \prod_{j=iN}^{(i+1)N-1} a_j.
	\end{equation}
	Notice also that for every $iN\leq m<(i+1)N$, we have 
	\begin{equation}\label{3.2}
	\prod_{j=iN}^{m}a_j\leq \left\|Dp \right\|^N,
	\end{equation} where $\left\|Dp \right\|$ is the uniform norm of $Dp$ on the Julia set $J(p)$.
	\medskip
	\par  Combining (\ref{3.1}) and (\ref{3.2}), for $kN\leq n<k(N+1)$ we have
	\begin{equation*}
	r_n=r_0\left( \prod_{i=0}^{k-1}\prod_{j=iN}^{(i+1)N-1}\frac{a_j}{\mu_j}\right) \prod_{j=kN}^{n}\frac{a_j}{\mu_j}\leq r_0 (1+\varepsilon_0)^{-kN}\frac{\left\|Dp \right\|^N }{(1+\varepsilon_0)^{n-kN}}=\left\|Dp \right\|^N  r_0(1+\varepsilon_0)^{-n}.
	\end{equation*}
	Taking $C_2=\left\|Dp \right\|^N $ the first inequality is proved.
	\medskip

	\par To prove the second inequality, notice that if the block $\left[ iN, (i+1)N\right) $ is of first type, then 	 \begin{equation*}
	\prod_{j=iN}^{(i+1)N-1}\mu_j= (1+\varepsilon_0)^N \prod_{j=iN}^{(i+1)N-1} a_j.
	\end{equation*}

\par Assume that $\left\lbrace  i_0,i_2,\dots, i_{l-1}\right\rbrace \subset \left\lbrace 0,1,\dots,k-1  \right\rbrace $ are all the integers such that the block $ \left[ i_hN,(i_h+1)N\right)  $ is of second type, $0\leq h\leq l$, then we have
\begin{align}\label{3.3}
r_n&=r_0\left( \prod_{i=0}^{k-1}\prod_{j=iN}^{(i+1)N-1}\frac{a_j}{\mu_j}\right) \prod_{j=kN}^{n}\frac{a_j}{\mu_j} \notag\\&=r_0(1+\varepsilon_0)^{(l-k)N}\left( \prod_{h=0}^{l-1}\prod_{j=i_hN}^{(i_h+1)N-1}\frac{a_j}{\mu_j}\right) \prod_{j=kN}^{n}\frac{a_j}{\mu_j} \notag\\&\geq r_0(1+\varepsilon_0)^{(l-k)N}\left( \prod_{h=0}^{l-1}\prod_{j=i_hN}^{(i_h+1)N-1}\frac{a_j}{(1+\varepsilon_0)^2}\right) \prod_{j=kN}^{n}\frac{a_j}{C_2(1+\varepsilon_0)^2} \notag\\&\geq \frac{r_0}{C_2} (1+\varepsilon_0)^{-2n}\left(  \prod_{h=0}^{l-1} \prod_{j=i_hN}^{(i_h+1)N-1}a_j\right) \prod_{j=kN}^{n}a_j.
\end{align}
\par Since the block $ \left[ i_hN,(i_h+1)N\right)  $ is of second type, then necessarily there is an integer $j$ satisfies $i_hN\leq j<(i_h+1)N$ and $ \dist(p^j(v),\Crit' )\leq \eta $. By Lemma \ref{hyperbolic away}, the product of derivative between two points $x_{n_1}$, $x_{n_2}$ such that $ \dist(x_{n_i},\Crit' )>\eta $, $n_1<i<n_2$ satisfies 
\begin{equation}\label{3.6}
 \prod_{i=n_1+1}^{n_2-1}a_j\geq C_1\eta^{\alpha}.  
\end{equation}
Notice that the number of such maximal blocks $[x_{n_1},x_{n_2}] $ in $ \left[ i_hN,(i_h+1)N\right)  $  is equal to the cardinality $\#\left\lbrace j\in \left[ i_hN,(i_h+1)N\right):\dist (p^j(v),\Crit')\leq \eta\right\rbrace +1$. 
\medskip 

\par There is also a constant $C_4>0$ such that $ \dist (p^j(v),\Crit' )\leq \eta $ implies $a_j\leq C_4\eta$ for $\eta$ small. Thus by choosing $\eta$ small, for $j$ satisfying $\dist (p^j(v),\Crit')\leq \eta$ we have
\begin{equation}\label{3.5}
C_1\eta^{\alpha}\geq C_1 \left( \frac{a_j}{C_4}\right) ^{\alpha}\geq Ca_j^{\alpha}\geq a_j^{1+\alpha}.
\end{equation}
\par Combining (\ref{3.5}) and (\ref{3.6}) we get
\begin{align}\label{3.4}
\left(  \prod_{h=0}^{l-1} \prod_{j=i_hN}^{(i_h+1)N-1}a_j\right) \prod_{j=kN}^{n}a_j&\geq   \prod_{\begin{subarray}{c}j=0\\d(p^j(v),\Crit' )\leq \eta\end{subarray}}^{n}C_1\eta^{\alpha} a_j\geq \prod_{j=0}^n a_j^{2+\alpha}\geq   e^{\left( -n\iota-C_0\right) \left( \alpha+2 \right) }.
\end{align}

\par Combining (\ref{3.3}) and (\ref{3.4}) we get
\begin{equation*}
r_n\geq \frac{r_0}{C_2} (1+\varepsilon_0)^{-2n}e^{\left( -n\iota-C_0\right) \left( \alpha+2 \right) }.
\end{equation*}
Setting $C_3:=e^{-\left( \alpha+2 \right)C_0}/C_2$ the conclusion follows.
\end{proof}
\medskip
 \par The following proposition clearly implies Theorem \ref{manifold} in the case $p$ satisfying TCE and WR.
 \medskip
 \par We let $ U_i:=D(p^i(v),r_0(1+\varepsilon_0)^{-3i}) $, $ V_i:=D(p^i(v),r_i) $ and let $B_i=U_i\times V_i$ for every positive integer $i$.
 
 \begin{proposition} \label{skew-stable}
 Assume $p$ satisfies TCE and WR($ \eta,\iota $) with small $\iota$ (to be determined in the proof). Then there exist $r_0>0$ such that for arbitrary $v\in J(p)\cap CV(p)$ and every $i$, the map $f:B_i\to B_{i+1}$ is a degree 1 crossed mapping and satisfies the condition of Theorem \ref{s-manifold}. As a consequence there is a stable manifold at $v$. 
 \end{proposition}

\begin{proof}
	We first show that for carefully chosen $\iota$ and for $r_0$ sufficiently small, $f$ restricted to a neighborhood of $ \overline{B_n} $ is injective for every $n$.  The WR condition implies the following Slow Recurrence property: there exist a small $\alpha(\iota)>0$, such that $\dist(p^n(v),\Crit' )>e^{-n\alpha}$ for all large  $n$, see Lemma A.2 for the proof. We let $\iota$  sufficiently small so that $\alpha<<\log (1+\varepsilon_0)$. By Lemma \ref{size-disk} (1) we can let $r_0$ small such that $r_n<< e^{-\alpha n}$  and also $r_0 (1+\varepsilon_0)^{-3n}<<e^{-\alpha n}$, for every $n$ .   
	\medskip
	\par We need the following general fact: if $f:W\to f(W)$ is a proper holomorphic map satisfying no critical points and $f(W)$ is simply connected, then $f$ is injective. Thus in our case, to show $f$ restricted to a small neighborhood of $ \overline{B_n} $ is injective,  it is sufficient to show $f(\overline{B_n})$ is contained in a simply connected domain which is disjoint with the critical value set of $f$.  Let $ M=\sup_{x\in\Omega} \left( \left| \frac{\partial f}{\partial z}\right| ,\;\left| \frac{\partial f}{\partial t}\right|\right)   $, where $\Omega$ is some compact subset such that $\Delta\times\mathbb{C}\setminus \Omega$ is in the basin of $\infty$.  Then $f(\overline{B_n})\subset U_{n+1} \times D(p^{n+1}(v), 2Mr_n) $. Let $l$ be the maximal order of the critical points in $J(p)$, then there is a constant $C>0$ such that $\text{dist} (p^{n+1}(v), p(Crit'))\geq C e^{l\alpha n}$. By choosing sufficiently small $r_0$ and $\iota$ we get that $U_{n+1} \times D(p^{n+1}(v), 2Mr_n)$ is disjoint with the critical value set of $f$. Thus $f$ restricted to a small neighborhood of $ \overline{B_n} $ is injective.
	\medskip
	\par Next we prove $f:B_n\to B_{n+1}$ is a degree $ 1 $ H\'enon-like map for every $n$.  By $|\lambda|<(1+\varepsilon_0)^{-3}$ we get   $ \pi_1(f(\overline{B_n}))\subset \pi_1(B_{n+1})  $, thus $ f(\overline{B_n})\cap \partial_v B_{n+1}=\emptyset $, thus it is easy to verify $ f(\overline{B_n})\cap \partial B_{n+1}\subset \partial_h B_{n+1}$.
	\medskip
	\par  To prove $ f(\partial_h B_n)\cap\overline{B_{n+1}}=\emptyset $, we first show that if $r_0$ is sufficiently small, and if  we let $\hat{V}_{n}:=D(p^{n+1}(v),(1+\varepsilon_0)^{\frac{1}{2}}r_{n+1})$, then  $\hat{V}_n\subset p(V_n)$ for every $n$. To see this, First by the definition of $r_n$ we have 
	\begin{equation*}
	a_n r_n=\mu_n r_{n+1}\geq (1+\varepsilon_0)r_{n+1},
	\end{equation*}where $a_n=|p'(p^n(v))|$. We choose sufficiently small $r_0$ such that $r_n\ll \dist (p^n(v),\Crit') $. Then the Koebe distortion theorem gives us $D(p^{n+1}(v),(1+\varepsilon_0)^{\frac{1}{2}}r_{n+1})\subset p(V_n)$ as desired.
	\par For every point $x\in \partial_h B_n$, let $y=\pi_2(x)$, by the above result we have $\dist (p(y),V_{n+1})\geq \left( (1+\varepsilon_0)^{1/2}-1\right) r_{n+1}$, thus we get
	\begin{align*}
	 \dist _v(f(x),B_{n+1})&\geq \left( (1+\varepsilon_0)^{1/2}-1\right) r_{n+1}-\dist _v(f(x),p(y))\\&\geq \left( (1+\varepsilon_0)^{1/2}-1\right) r_{n+1}- Mr_0(1+\varepsilon_0)^{-3n}\;(\text{by mean value theorem})\\&\geq \left( (1+\varepsilon_0)^{1/2}-1\right)  C_3r_0 e^{-\left( \alpha+2 \right)n\iota}(1+\varepsilon_0)^{-2n}-Mr_0(1+\varepsilon_0)^{-3n},
	\end{align*}
	where  $ M=\sup_{x\in\Omega} \left( \left| \frac{\partial f}{\partial z}\right| ,\;\left| \frac{\partial f}{\partial t}\right|\right)   $ as before. By choosing $\iota\ll\varepsilon_0$ we get $ \dist _v(f(x),B_{n+1})>0 $. Thus $ f(\partial_h B_n)\cap\overline{B_{n+1}}=\emptyset $.
	\medskip
	\par It is easy to show $f:B_n\to B_{n+1}$ has degree 1. The reason is that the forward image of a vertical disk is again a vertical disk, and $ \pi_2\circ f $ is of degree 1 when restricted to this vertical disk. Since $f$ keep the degree of the curve fix, $f$ must have degree 1.
	\medskip
	\par Finally we set $ V'_n:=p^{-1} V_{n+1}$ and show that the modulus of the annulus $ V_n-\overline{V'_n} $ is uniformly bounded from below. Since the modulus is invariant under univalent maps, we have that
	\begin{equation*}
		\Mod (V_n\setminus V'_n)=\Mod \left( p(V_n)\setminus V_{n+1}\right) \geq \Mod \left( \hat{V}_n \setminus V_{n+1}\right) =\frac{1}{4\pi}\log(1+\varepsilon_0).
	\end{equation*} 

   \par Now all the conditions in Theorem \ref{s-manifold} are checked, we conclude that there is a stable manifold in the sense of Theorem 3.7. Since the dynamics contracts exponentially transverse to $L$, this is a stable manifold in the usual sense.
\end{proof} 
\medskip
\subsection{The Positive Lyapunov case}
\par Next  we assume $p$ satisfies Positive Lyapunov instead of TCE+WR($ \eta,\iota $). We can then construct the stable manifold by arguing as before. Indeed, let $\chi_v$ be the following vertical Lyapunov exponent
\begin{equation}\label{limit}
\chi_v:=\lim_{n\to\infty}\frac{1}{n}\log|(p^n)'(v)|>0.
\end{equation}
We  need to construct $\mu_n$ such that an estimate of $r_n$ similar to that of Lemma \ref{size-disk} holds, and also show  the Slow Recurrence property: for every $\alpha>0$, $\dist (p^n(v),\Crit' )>e^{-n\alpha}$ for all large $n$. That Positive Lyapunov implies Slow Recurrence is shown in Lemma A.3. 
\medskip
\par To show an estimate of $r_n$ similar to Lemma \ref{size-disk}, we define $\mu_n:=(1+\varepsilon_0)e^{\chi_v}$ for every $n$, and let $r_n:=r_0\prod_{i=1}^{n-1}\frac{a_i}{\mu_i}$, where $r_0>0$ is a constant, $ a_i:=|p'(p^i(v))| $.
\begin{lemma}\label{size-disk 2}
There are constants $C_2,C_3>0$ such that the following estimates of $r_n$ hold for $n\geq 0$: 
\medskip

(1) $r_n\leq C_2r_0(1+\frac{\varepsilon_0}{2})^{-n}$.
\medskip  

(2) $ r_n\geq C_3r_0(1+2\varepsilon_0)^{-n}$.

\end{lemma}
\begin{proof}
Since (\ref{limit}) holds, for every $\varepsilon>0$ small, there exist constants $C_2>0$ and $C_3>0$ such that $\prod_{i=1}^{n-1}a_i\leq C_2(1+\varepsilon)^ne^{n\chi_v}$ and $\prod_{i=1}^{n-1}a_i\geq C_3(1-\varepsilon)^ne^{n\chi_v}$ for every $n\geq 0$. Now it is sufficient to choose $\varepsilon= \frac{\varepsilon_0}{2+\varepsilon_0}$.
\end{proof}

\begin{remark}
	The above proof actually only use that the ratio of limsup and liminf of (\ref{limit}) is sufficiently close to 1, thus the condition Positive Lyapunov can be replaced by a weaker condition (ratio of upper and lower Lyapunov exponent is sufficient close to 1) to make the main theorem hold.
\end{remark} 

\par Finally the same argument as in Proposition \ref{skew-stable} gives the existence of a local stable manifold at $v$, in case that $p$ satisfies Positive Lyapunov. Thus the proof of Theorem \ref{manifold} is complete.
\medskip
\subsection{Renormalization map}
\par Next we introduce the renormalization map of a critical value variety $\mathcal{V}$. 	We let $\mathcal{V}$ be a component of the critical value variety that defined in a neighborhood of $B_0$, where $B_0$ is as in Proposition 3.11. We assume that the germ of $\mathcal{V}$ at $ (0,v) $ does not coincide with $W_{\loc }^s(v)$ (the converse hypothesis that $\mathcal{V}=W_{\loc }^s(v)$ makes the main theorem even easier to prove, we will see this later). Then by the definition of the stable manifold in Theorem \ref{s-manifold}, for $N$ large $f^N(\mathcal{V})\not\subset B_N$, thus $f^N(\mathcal{V})\cap \partial B_N\neq \emptyset$. By the definition of a H\'{e}non-like map we must have $f^N(\mathcal{V})\cap \partial B_N\subset \partial_h B_N$. Thus for $N$ large, $ f^N(\mathcal{V}) $ is a degree $d$ vertical curve. Note that $d$ is constant  because our H\'{e}non-like maps have degree 1.  Without loss of generality,  we may assume that $\mathcal{V}$ is a degree $d$ vertical curve (in $B_0$), otherwise we may replace $\mathcal{V}$ by some $f^N(\mathcal{V})$. By the definition of degree 1 crossed mapping, for every $n\geq 0$, $f^n(\mathcal{V})\cap B_n$ is also a degree $d$ vertical curve (in $B_n$).
\medskip
\par We assume $\mathcal{V}$ has the  parametrization $\mathcal{V}=\left\lbrace \gamma(t):t\in D(0,r_0) \right\rbrace $  of the form $\gamma(t)=(t^l,\psi(t))$, where $r_0$ is the radius of $U_0$ in Proposition \ref{skew-stable}, $l$ is a positive integer and $\psi$ is holomorphic. Since $\mathcal{V}$ is a vertical curve, we can further assume that $\psi'(0)\neq 0$, for otherwise $\mathcal{V}$ can not be a vertical curve in $B_0$ when $r_0$ is sufficiently small.
\medskip
\par For $n\geq 0$ we let $W_n:=f^{-n}(B_n)\cap B_0$. The map 
\begin{equation*}
\psi_n:=\pi_2\circ f^n\circ \gamma:\gamma^{-1}(W_n)\to V_n
\end{equation*} is well defined. 
\begin{definition}\label{rho}
For every integer $n\geq 0$, let $\rho_n$ be the maximal positive real number such that $\psi^{-1}_n(\frac{1}{2}V_n)$ contains a disk $D(0,\rho_n)$, where $\frac{1}{2}V_n$ denotes a disk centered at the same point of $V_n$ but with one-half radius.	
\end{definition}
 Concretely, $\rho_n$ is the typical size of the piece of $\mathcal{V}$ which remains in $ B_j $ under iteration up to time $n$. We can then define the renormalization map as follows
\begin{definition}
For every integer $n\geq 0$, the $n$-th renormalization map $\phi_n$ is the holomorphic map from $D_n$ to $\mathbb{C}$, defined by $\phi_n(z)= \psi_n\circ L_{\rho_n}(z)$ for $z\in D_n$,
where $D_n:=L_{\rho_n}^{-1}(\psi_n^{-1}(\frac{1}{2}V_n))$ is a domain in $\mathbb{C}$, and  for $r\in\mathbb{C}$,  $L_r:\mathbb{C}\to \mathbb{C}$ denotes the linear map $L_r(z)=r z$.
\end{definition}
 By the definition of $D_n$ we know that $D(0,1)\subset D_n$. The following proposition is crucial.
\begin{proposition}\label{renorm}
The renormalization map $\phi_n$ has uniformly bounded (topological) degree. Moreover there exist a constant $C_0>0$ such that $\diam D_n\leq C_0$, and $\rho_n$ is exponentially small, namely $\rho_n\leq C\mu_{CE}^{-n} r_n$, where $\mu_{CE}>1$ and $C>0$ are constants.
\end{proposition}
\begin{proof}
The constant $\mu_{CE}$ corresponds to the constant appeared in the CE condition (Definition \ref{CE}). We will see in the proof that TCE+WR($ \eta,\iota $) with $\iota$ small implies CE. By our construction, $f^n\circ \gamma\left( \gamma^{-1}(W_n)\right) $ is a degree $d$ vertical curve (in $B_n$) for every $n$. Thus $\psi_n$ is of uniformly bounded degree and so is $\phi_n$.
\medskip
\par To show $\diam D_n\leq C_0$, we observe that by Theorem \ref{Koebe} (2.2)  there exist a uniform constant $C_0>0$ such that $C_0 \rho_n\geq \diam (\psi_n^{-1}(\frac{1}{2}V_n))$, thus 
\begin{equation*}
\diam D_n=\frac{1}{\rho_n} \diam \left( \psi_n^{-1}\left( \frac{1}{2}V_n\right) \right) \leq C_0.
\end{equation*}
\medskip
\par To show that $\rho_n$ is exponentially small, first notice that TCE + WR($ \eta,\iota $) with $\iota$ small implies CE (see Lemma A.4), and also Positive Lyapunov implies CE. Thus $ |(p^n)'(v)| $ is exponentially large, that is there exist $C>0$ such that $ |(p^n)'(v)| \geq C \mu_{CE}^n$ with $\mu_{CE}$ slightly smaller than $\mu_{Exp}$. Thus 
\begin{equation*}
|\psi'_n(0)|=|\psi'(0)||(p^n)'(v)|\geq C \mu_{CE}^n |\psi'(0)|,
\end{equation*} which is exponentially large. By Theorem \ref{Koebe} (2.1) we get  $\rho_n\leq C \mu_{CE}^{-n} r_n$, which is exponentially small.
\end{proof}
\medskip

	\section{Slow approach to Crit'}
	In this section our aim is to prove Theorem 1.2. First we remark that it is not true that for every vertical fiber $ \left\lbrace t=t_0\right\rbrace $  Lebesgue a.e. (in the sense of one-dimensional Lebesgue measure) point in this fiber slowly approach Crit', as pointed out in \cite{peters2016polynomial}. Indeed in \cite{peters2016polynomial} the authors construct a vertical Fatou disk which comes exponentially close to Crit'. Instead, we need to select a full measure family of vertical fibers such that Lebesgue a.e. point in the fiber slowly approach Crit'. This will be proved by studying the renormalization maps along critical varieties constructed in the previous section. Together with Przytycki's lemma we can actually  track the orbits of points in Crit (Lemma 4.2 and Lemma 4.3). Thus we need the non-uniform hyperbolic conditions in section 3 to make sure that the stable manifolds at each $ v\in CV(p)\cap J(p) $ exist. 
	
	\medskip

	\par Let $W_0$ be a  forward invariant open subset of $F(p)$ satisfying $\overline{p(W_0)}\subset W_0$. Such a $W_0$ exist since $F(p)$ is a union of attracting basins. Let $W_m=p^{-m}(W_0)$, and let $ K_m $ be the complement $ K_m=\mathbb{C}\setminus W_m $. 
	\begin{lemma}\label{measure shrink}
		If $p$ satisfies TCE, then the Lebesgue measure of $K_m$ decreases exponentially fast with $m$.
	\end{lemma}
	\begin{proof}
		We may assume that $K_0$ is sufficiently close to $J(p)$, that is,
				\begin{equation*}
		K_0\subset \left\lbrace x:\;\dist (x,J(p)) \leq r \right\rbrace,
		\end{equation*} where $r$ is the constant appearing in the definition of the TCE condition. Thus by definition of the TCE condition we have
		\begin{equation*}
		K_m\subset \left\lbrace x:\;\dist (x,J(p)) \leq \mu_{Exp  }^{-m}\right\rbrace 
		\end{equation*}
		for $m\geq 0$.

	 \par Denote $N_\varepsilon$ the number of $ \varepsilon $-disks needed to cover $ J(p) $, by Proposition \ref{TCE} (2), the Minkowski dimension of $J(p)$ is $ h<2$. Hence for $ \varepsilon>0 $ sufficiently small we have $N_\varepsilon<\varepsilon^{-h}$. Choosing $\varepsilon=\mu_{Exp}^{-m}$, for $m$ sufficiently large, $ K_m $ is covered by at most $ \mu_{Exp  }^{mh} $ disks of radius $ \mu_{Exp  }^{-m} $. Thus the measure of $K_m$ is at most $\pi \mu_{Exp  }^{m(h-2)} $, and the conclusion follows.
		
	\end{proof}
    \medskip
	\par In the following we assume $p$ satisfies TCE+WR or Positive Lyapunov. Let $\Delta=D(0,r_0)$. The set $ W= \Delta\times W_0$ is contained in $ F(f) $, and $\overline{f(W)}\subset W$, let $ K=\left( \Delta\times\mathbb{C}\right) \setminus W $. Let $ W' $ be  the $ \varepsilon- $neighborhood of $ W $,
	\begin{equation*}
	W':=\left\lbrace x\in\Delta\times\mathbb{C}:\dist (x,W)<\varepsilon\right\rbrace.
	\end{equation*}
	For $ \varepsilon $ sufficiently small $W'$ is forward invariant.
	\medskip

	\par Let $ K'=\left( \Delta\times\mathbb{C}\right) \setminus W' $. By using Proposition \ref{renorm}, we show that for most vertical fibers, the critical points on the fiber  move to the Fatou set fairly quickly. The argument is similar to Peters-Smit \cite{peters2018fatou} who treated the sub-hyperbolic case.
	\medskip
	\par  Let us choose a critical value variety $\mathcal{V}$ passing through $v\in CV(p)\cap J(p)$, parametrized as before: $\mathcal{V}=\left\lbrace (t^l,\psi(t)):t\in \Delta\right\rbrace$. Let $\phi_n$ be the renormalization map defined in subsection 3.4.
	\medskip
	\par  For every integer $ s\geq 0 $, we define $j(s)$ to be the maximal integer such that $|\lambda^s|\leq \rho_{j(s)} $, where $\rho_n$ is as in Definition \ref{rho}.
	\begin{lemma}\label{escaping}
		For every critical value variety $\mathcal{V}$ passing through $ v\in CV(p)\cap J(p)$ that does not coincide with the stable manifold at $v$, there is a full Lebesgue measure subset $E_v\subset \Delta$ such that for every $u\in E_v$ there exist an integer $ N_u $ and $\beta>0$ independent of $u$ such that for every $s\geq N_u$, we have
		
		\begin{equation*}
		f^{j(s)+ \beta s }(\gamma(\lambda^{s} u))\in W'.\;\text{where}\;\gamma(t)=(t^l,\psi(t)).
		\end{equation*}
	\end{lemma}
\medskip
We note that there is some abuse in notation. For the simplicity when we write a non-integer number $s$ as an iteration number of a map $f$, we mean the iteration of $\lfloor s\rfloor$ times.
	\begin{proof}
		Fix $\beta>0$ arbitrary for the moment. For every integer $s\geq 0$,	let $A_s$ be the set
		\begin{equation*}
		A_s=\left\lbrace u\in \Delta:f^{j(s)+ \beta s }(\gamma(\lambda^{s} u))\in K'\right\rbrace .
		\end{equation*}
		
		By the definition of the renormalization map $\phi_n$, we have
		\begin{equation*}
		A_s=\left\lbrace u\in \Delta:f^{ \beta s }\left( \left( \lambda^su\right) ^l,\phi_{j(s)}\left( \frac{1}{\rho_{j(s)}}\lambda^s u\right) \right) \in K'\right\rbrace .
		\end{equation*}
	
		\par Let $ M=\sup_{x\in\Omega} \left( \left| \frac{\partial f}{\partial z}\right| ,\;\left| \frac{\partial f}{\partial t}\right|\right)   $, where $\Omega$ is a compact subset such that $\Delta\times\mathbb{C}\setminus \Omega$ is in the basin of $\infty$. By a shadowing argument there exists $ C>0 $ such that for every integer $m\geq 0$,  if $ f^m(x)\in K' $ and $ |\pi_1(x)|< CM^{-m}$,  then $ \pi_2(x)\in K_m $.  It is equivalent to say that if $ x\in\Delta\times\mathbb{C}  $ satisfies $ |\pi_1(x)|< CM^{-m}$ and $ \pi_2(x)\in W_m $,  then $f^m(x)\in W' $. 
		\medskip
		\par We choose $ \beta $ sufficiently small, such that for large enough $s$ we have 
		\begin{equation*}
		\left| \left( \lambda^s u\right)  ^l\right|   <CM^{- \beta s }\;\left( \beta<\frac{-\log|\lambda|}{\log M}\;\text{is enough}\right) .
		\end{equation*} 
	
		\par Thus we get
		\begin{equation*}
		A_s\subset \left\lbrace u\in \Delta: \phi_{j(s)}\left( \frac{1}{\rho_{j(s)}}\lambda^s u\right) \in K_{ \beta s }\right\rbrace.
		\end{equation*}

		\par Next we estimate the measure of the slightly bigger set 
		\begin{align*}
			\widetilde{A}_s&:=\left\lbrace u\in \Delta: \phi_{j(s)}\left( \frac{1}{\rho_{j(s)}}\lambda^s u\right) \in K_{ \beta s }\right\rbrace=\rho_{j(s)}\lambda^{-s}\phi_{j(s)}^{-1}(K_{ \beta s }).
		\end{align*}
	
	\par  By Lemma \ref{measure shrink} the Lebesgue measure of $ K_{ \beta s } $ decreases exponentially with $s$.  Next we prove that we can choose $\varepsilon_0$ and $\iota$ in Lemma \ref{size-disk} (2) sufficiently small so that the ratio
	\begin{equation*}
	\frac{\meas (K_{ \beta s }\cap \frac{1}{2}V_{j(s)})}{r^2_{j(s)}}
	\end{equation*} is exponentially small.   Since
	\begin{equation*}
 \frac{\meas (K_{ \beta s }\cap \frac{1}{2}V_{j(s)})}{r^2_{j(s)}}\leq 	\frac{\meas (K_{ \beta s })}{r^2_{j(s)}},
	\end{equation*} 
	it is sufficient to show $\meas (K_{ \beta s })/{r^2_{j(s)}}$ is exponentially small. 
	\medskip
	\par If we choose $\beta=\frac{-\log|\lambda|}{2\log M}$, then by Lemma \ref{measure shrink} (again for large enough $s$)
	\begin{equation}\label{meas K}
	\meas (K_{ \beta s })\leq \pi\mu_{Exp}^{\frac{s\log|\lambda|(2-h)}{2\log M}}.
	\end{equation}
	\par On the other hand, by Proposition \ref{renorm} we have 
	\begin{equation*}
	\rho_{j(s)}\leq C\mu_{CE}^{-j(s)}r_{j(s)}\leq \mu_{CE}^{-j(s)}\;\; (\text{since  $r_0$ small}).
	\end{equation*}
	\par Then by the definition of $j(s)$ we have
	\begin{equation*}
	j(s)\leq \frac{-s\log |\lambda|}{\log \mu_{CE} }.
	\end{equation*} 
	Then by Lemma \ref{size-disk} (2) we have
	\begin{equation}\label{meas r}
	r_{j(s)}\geq C_3r_0 e^{-(\alpha+2)j(s)\iota}(1+\varepsilon_0)^{-2j(s)}\geq C_3r_0 e^{\frac{(\alpha+2)\log |\lambda|}{\log \mu_{CE} }\iota s}(1+\varepsilon_0)^{\frac{2\log |\lambda|s}{\log \mu_{CE} }} .
	\end{equation}
	\par By (\ref{meas K}) and (\ref{meas r}) we can choose $\varepsilon_0$ and $\iota$ sufficiently small such that $\meas (K_{ \beta s }\cap \frac{1}{2}V_{j(s)})/r^2_{j(s)}$ is exponentially small.
	\medskip 
	\par It is proved in Proposition \ref{renorm} that the map $ \phi_n $ has uniformly bounded degree, so by Lemma \ref{Koebe} (2.4) the measure of  $ \phi_{j(s)}^{-1}(K_{ \beta s }\cap \frac{1}{2}V_{j(s)}) $ also decreases exponentially with $s$. Finally 	since $\rho_{j(s)}\lambda^{-s} $ is uniformly bounded with $s$, the Lebesgue measure of $\widetilde{A}_s$ also decreases exponentially with $s$. Thus $ \sum_{s=1}^{\infty}\meas (A_s)\leq \sum_{s=1}^{\infty}\meas (\widetilde{A}_s)<\infty $. By the Borel-Cantelli lemma, there is a full measure subset $ E_v\in \Delta $ such that for every $ u\in E_v $, there exist an integer $ N_u $ such that when $s\geq N_u$, $u\notin A_s$. In other words, $ f^{j(s)+ \beta s }(\gamma(\lambda^{s} u))\in W' $ and the conclusion follows.
	\end{proof}

\par In the case where the critical value variety $\mathcal{V}$ passing through $ v\in CV(p)\cap J(p)$  coincides with the stable manifold of $v$, every $ y\in \mathcal{V} $ will shadow $v$ forever. Thus we get an estimate of the returning time to Crit' of   $ y $, simply by  Przytycki's lemma (Lemma \ref{Przytycki})

\begin{lemma}\label{degenerate}
  Assume that the critical value variety $\mathcal{V}$ passing through $ v\in CV(p)\cap J(p)$ coincides with the stable manifold of $v$.  Let $\gamma:\Delta\to \mathbb{C}^2$ be the parametrization of this stable manifold such that $\gamma(0)=v$, $\gamma(t)=(t,\psi(t))$ where $\psi$ is holomorphic.
 \par Let $\mathcal{C}$ be the critical variety of $f$ such that $f(\mathcal{C})=\mathcal{V}$. Then for every fixed $\alpha>0$ there exists a constant $K(\alpha)>0$ such that for every $ n\geq 0 $, $0\leq s\leq n$ and $u\in \Delta$, if $\dist _v(f^n(\gamma(u)),\mathcal{C})\leq e^{-\alpha n}$ and $\dist _v (f^{n-s}(\gamma(\lambda^{s} u)),\mathcal{C})\leq e^{-\alpha n}$, then $s\geq K n$.
\end{lemma}
\begin{proof}
	  We let $L_u$ to be the vertical line $\left\lbrace t=u\right\rbrace$. By the construction of bi-disks in Proposition \ref{skew-stable} , there exist constants $C_0>0$, $\lambda_1<1$ such that for every $n\geq 0$ and $u\in \Delta$ we have $\dist _v(f^n(\gamma(u)),p^n(v))\leq C_0 \lambda_1^n.$ Together with  $\dist _v(f^n(\gamma(u)),\mathcal{C})\leq e^{-\alpha n}$ we get  $\dist_v(p^n(v),\mathcal{C}\cap L_{\lambda^n u})\leq e^{-\alpha n}+C_0 \lambda_1^n$.
\medskip
\par For  similar reasons we have $\dist _v(f^{n-s}(\gamma(\lambda^su)),p^{n-s}(v))\leq C_0 \lambda_1^n.$ Together with the inequality  $\dist _v (f^{n-s}(\gamma(\lambda^{s} u)),\mathcal{C})\leq e^{-\alpha n}$ we get  $\dist_v (p^{n-s}(v),\mathcal{C}\cap L_{\lambda^n u})\leq e^{-\alpha n}+C_0 \lambda_1^n$.
\medskip
\par   On the other hand there exist $C_1>0$, $l'>0$ such that $\dist_v (\mathcal{C}\cap L_{\lambda^n u},c_0)\leq C_1|\lambda|^{n/l'}$ ($l'$ is related to the multiplicity of $\mathcal{C}$ at $c_0$), where $c_0=\mathcal{C}\cap L$ is the unique intersection point of $\mathcal{C}$ and the invariant line $L$. Then by the triangle inequality we have $\dist (p^n(v),c_0)\leq e^{-\alpha n}+C_0 \lambda_1^n+C_1 |\lambda|^{n/l'}$ and also $\dist (p^{n-s}(v),c_0)\leq e^{-\alpha n}+C_0 \lambda_1^n+C_1 |\lambda|^{n/l'}$. Thus $s$ is a return time of $p^{n-s}(v)$ into the small neighborhood $D(c_0,e^{-\alpha n}+C_0 \lambda_1^n+C_1 |\lambda|^{n/l'})$ of $c_0$. By Przytycki's lemma (Lemma \ref{Przytycki}) we get 
\begin{equation*}
s\geq -C\log  (e^{-\alpha n}+C_0 \lambda_1^n+C_1 |\lambda|^{n/l'}):=Kn,
\end{equation*} the conclusion follows.

\end{proof}
\medskip
	\par The main result of this section is the following equivalent form of Theorem 1.2,
	\begin{theorem}\label{slow approach}
		There is a full Lebesgue measure subset $ E\subset\Delta $ such that for every $ u\in E $,  for Lebesgue a.e. $x$ in the fiber $ L_u:\left\lbrace  t=u \right\rbrace $, $x$ slowly approach Crit'.
	\end{theorem}
\begin{proof}
It is enough to prove that for each fixed $ \alpha>0$ and $u\in E$, the set of points in $L_u$ satisfing $ \dist _v (f^n(x),\Crit' ) \geq e^{-\alpha n} $ for all large $n$ has full Lebesgue measure in $L_u$. We let $E$ be the intersection of all $E_v$, where $E_v$ is in Lemma \ref{escaping}, and $v$ ranges on the set of critical values. Thus $E$ has full Lebesgue measure in $\Delta$. For every $u\in E$ we consider the sets
\begin{equation*}
 E_n:=\bigcup_{c\in \Crit' \cap L_{\lambda^n u}}f^{-n}\left( D_v(c,e^{-\alpha n})\right),\;\text{and}\;E'_n:=\bigcup_{c\in \Crit' \cap L_{\lambda^n u}}f^{-n}\left( D_v(c,e^{-2\alpha n})\right).
\end{equation*} 

\par  (Recall that $D_v$ stands for vertical disk). For an arbitrary critical point $ c\in\Crit' \cap L_{\lambda^n u} $, we let $\Gamma$ be an arbitrary connected component of $f^{-n}\left( D_v(c,e^{-\alpha n})\right)  $.
\medskip
\par{ \em Step 1}, we show that the cardinality  $\#\left\lbrace  0\leq s\leq n:\;f^s(\Gamma)\cap \C \neq \emptyset\right\rbrace $ is uniformly bounded with respect to $n$. 
\medskip
\par For $n$ large enough $f^s(\Gamma)$ has no intersection with any critical variety $\mathcal{C}$ such that $\mathcal{C}\not\subset\Crit' $. The reason is the following. Take the radius of $ \Delta $ sufficiently small to make sure that $\mathcal{C}\subset\subset F(f)$. Thus if $ c'\in \mathcal{C}\cap f^s(\Gamma) $ for some $ 0\leq s\leq n $, then $\dist _v(f^{n-s}(c'),J(p))>\delta$ for some uniform constant $\delta$. On the other hand $\dist _v(c,J(p))\leq C|\lambda|^{\frac{n}{l}}$, where $c$ is as in the definition of $E_n$ and $E'_n$. 	Thus $\dist _v(f^{n-s}(c'),c)>\delta'$ for some  uniform constant $\delta'$, this is impossible when $n$ large since $f^{n-s}(c')\in D_v(c,e^{-\alpha n})$. Thus it is sufficient to show that the cardinality  $\#\left\lbrace  0\leq s\leq n:\;f^s(\Gamma)\cap \Crit' \neq \emptyset\right\rbrace $ is uniformly bounded with respect to $n$. For this it is sufficient to show that $\#\left\lbrace  0\leq s\leq n:\;f^s(\Gamma)\cap \mathcal{C}\neq \emptyset\right\rbrace $ is uniformly bounded with respect to $n$ for every local component of critical variety $\mathcal{C}\subset \Crit' $. 
\medskip
\par Now there are two cases. Let $ \mathcal{V}=f(\mathcal{C}) $ be a critical value variety, and let $v$ be the unique intersection point of $ \mathcal{V} $ and $ L $, $ v\in CV(p)\cap J(p) $. In the first case we assume that $\mathcal{V}$ does not coincide with the stable manifold at $v$ as in Lemma \ref{escaping}. We claim that if $n$ is large, $s$ satisfies $s+1+j(s+1)+\beta (s+1)\leq n$ and $s\geq N_u$, then we have $f^s(\Gamma)\cap\mathcal{C}=\emptyset$. For otherwise if $ c'\in f^s(\Gamma)\cap\mathcal{C} $ then $v':=f(c')\in  f^{s+1}(\Gamma)\cap \mathcal{V}$. Then by Lemma \ref{escaping} we have
$f^{j(s+1)+\beta (s+1)}(v')\in W'$. Since $W'$ is forward invariant, when $n-s-1\geq j(s+1)+\beta (s+1)$ implies $f^{n-s-1}(v')\in W'$, thus $ f^{n-s}(c')\in W' $. By the definition of $\Gamma$ we also have  $ \dist_v(f^{n-s}(c'), \mathcal{C})\leq e^{-\alpha n} $, which is a contradiction when $n$ large. To summarize,  there exist a uniform constant $0<\theta<1$ such that $f^s(\Gamma)\cap \mathcal{C}\neq \emptyset$ implies $s\geq \theta n$ or $s\leq N_u$. 
\medskip
\par We need the following Lemma.
\begin{lemma}\label{useful}
There is a constant $N=N(\theta)>0$ such that
\begin{equation*}
  \#\left\lbrace  \theta n\leq s\leq n:\;f^s(\Gamma)\cap \C \neq \emptyset\right\rbrace\leq N .  
\end{equation*}
\end{lemma}
\begin{proof}
We first show that the cardinality $\#\left\lbrace  (1-\kappa)n< s\leq n:\;f^s(\Gamma)\cap \mathcal{C}\neq \emptyset\right\rbrace $ is uniformly bounded with respect to $n$, where $\kappa$ is the constant defined by,
\begin{equation*}
\kappa=\min\left( \frac{-\theta\log|\lambda|}{4l\log M},\frac{1}{2}\right) \;\text{ with } \;M=\sup_{x\in\Omega} \left( \left| \frac{\partial f}{\partial z}\right| ,\;\left| \frac{\partial f}{\partial t}\right|\right) .
\end{equation*} 
Note that by the definition of $\kappa$
\begin{equation}\label{kappa}
M^{\kappa}\leq |\lambda|^{\frac{-\theta}{4l}}.
\end{equation}
\par Assume $s_1<s_2$ satisfy $f^{s_i}(\Gamma)\cap \mathcal{C}\neq\emptyset$ and $ \left( 1-\kappa\right) n< s_i\leq n$, $i=1,2$.  Let $c_i\in f^{s_i}(\Gamma)\cap \mathcal{C}$, $i=1,2$. Let $c_0=\mathcal{C}\cap L$ be the unique intersection point. Then we have
\begin{equation}\label{0}
\dist _v(f^{n-s_2}(c_2),c_0)\leq \dist _v(f^{n-s_2}(c_2),\mathcal{C}\cap L_{\lambda^{n-s_2}u})+\dist _v(\mathcal{C}\cap L_{\lambda^{n-s_2}u},c_0)\leq e^{-n\alpha}+C|\lambda|^{\frac{n}{l}}.
\end{equation}
Similarly
\begin{equation*}
\dist _v(f^{n-s_1}(c_1),c_0)\leq \dist _v(f^{n-s_1}(c_1),\mathcal{C}\cap L_{\lambda^{n-s_1}u})+\dist _v(\mathcal{C}\cap L_{\lambda^{n-s_1}u},c_0)\leq e^{-n\alpha}+C|\lambda|^{\frac{n}{l}}.
\end{equation*}
\medskip
\par By the definition of $\theta$ we also have $\dist _v(c_1,c_2)\leq C|\lambda|^{\frac{n\theta}{l}}$. Thus by (\ref{kappa}) we have  
\begin{align*}
\dist _v(f^{n-s_2}(c_1),f^{n-s_2}(c_2))\leq M^{n-s_2} C|\lambda|^{\frac{n\theta}{l}}\leq C|\lambda|^{\frac{3n\theta}{4l}}.\;(\text{By the choice of}\; s_2 ).
\end{align*} 

\par Let $ y=\pi_2(f^{n-s_2}(c_2)) $. Using (\ref{kappa}) again we have 
\begin{equation*}
\dist _v(f^{n-s_1}(c_1),p^{s_2-s_1}(y))\leq M^{s_2-s_1}\dist _v(f^{n-s_2}(c_1),f^{n-s_2}(c_2))\leq C|\lambda|^{\frac{n\theta}{2l}}.
\end{equation*}

\par Thus we have
\begin{equation}\label{s2-s1}
\dist(c_0,p^{s_2-s_1}(y))\leq \dist _v(f^{n-s_1}(c_1),c_0)+ \dist _v(f^{n-s_1}(c_1),p^{s_2-s_1}(y))\leq e^{-n\alpha}+C|\lambda|^{\frac{n\theta}{2l}}.
\end{equation}
\par Combining (\ref{0}) and (\ref{s2-s1}) we infer that $s_2-s_1$ is a return time of $y$ in the small disk $ D(c_0, e^{-n\alpha}+C|\lambda|^{\frac{n\theta}{2l}}) $, by Przytycki's lemma (Lemma \ref{Przytycki}) there exist a constant $K(\alpha)$ such that $s_2-s_1\geq Kn$. Thus $\#\left\lbrace  (1-\kappa)n< s\leq n:\;f^s(\Gamma)\cap \mathcal{C}\neq \emptyset\right\rbrace\leq \frac{\kappa}{K(\alpha)}+1 $ .
\medskip
\par By Lemma \ref{Koebe} (2.3) there exist $ \alpha_1>0 $ such that  $\diam f^{(1-\kappa)n}(\Gamma)\leq e^{-\alpha_1n}$.
\medskip
\par Next if we consider the cardinality $\#\left\lbrace  (1-2\kappa)n< s\leq (1-\kappa)n:\;f^s(\Gamma)\cap \mathcal{C}\neq \emptyset\right\rbrace $, we replace $ f^n(\Gamma) $ by $ f^s(\Gamma) $ where $s$ satisfies $(1-2\kappa)n<s\leq (1-\kappa)n$, $ f^s(\Gamma)\cap \mathcal{C}\neq \emptyset $ and is maximal. Repeating the same argument we know there is a constant $ K(\alpha_1)>0 $ such that $\#\left\lbrace  (1-2\kappa)n< s\leq (1-\kappa)n:\;f^s(\Gamma)\cap \mathcal{C}\neq \emptyset\right\rbrace\leq \frac{\kappa}{K(\alpha_1)}+1 $. After finitely many iteration of the argument we get $\#\left\lbrace  \theta n< s\leq n:\;f^s(\Gamma)\cap \mathcal{C}\neq \emptyset\right\rbrace $ is uniformly bounded with respect to $n$.
\end{proof} 
\medskip
\par We continue the proof of Theorem \ref{slow approach}.  Since we have $\#\left\lbrace  0\leq s\leq \theta n:\;f^s(\Gamma)\cap \mathcal{C}\neq \emptyset\right\rbrace\leq N_u $. Thus by Lemma \ref{useful} we get that $\#\left\lbrace  0\leq s\leq n:\;f^s(\Gamma)\cap \C \neq \emptyset\right\rbrace $ is uniformly bounded with respect to $n$.
\medskip
\par In the second case we assume that $\mathcal{V}=f(\mathcal{C})$  coincides with the stable manifold at $v$ as in Lemma \ref{degenerate}. We also want to show that  $\#\left\lbrace  0\leq s\leq n:\;f^s(\Gamma)\cap \C \neq \emptyset\right\rbrace $ is uniformly bounded with respect to $n$. Let as before $\gamma$ be the parametrization of the stable manifold. Assume $0\leq s_1<s_2\leq n$ satisfy $f^{s_i}(\Gamma)\cap \mathcal{C}\neq\emptyset$, let $c_i\in f^{s_i}(\Gamma)\cap \mathcal{C}$, $i=1,2$. Let $f(c_1)=\gamma(u_1)$, then $f(c_2)=\gamma(\lambda^{s_2-s_1}u_1)$. By the definition of $\Gamma$  we have $\dist _v(f^{n-s_1}(\gamma(u_1)),\mathcal{C}\cap L_{\lambda^{n-s_1}u_1})\leq e^{-\alpha n} $, and $ \dist _v(f^{n-s_2}(\gamma(\lambda^{s_2-s_1}u_1),\mathcal{C}\cap L_{\lambda^{n-s_1}u_1})\leq e^{-\alpha n} $. By Lemma \ref{degenerate} we have $s_2-s_1\geq Kn$. Thus  $\#\left\lbrace  0\leq s\leq n:\;f^s(\Gamma)\cap \C \neq \emptyset\right\rbrace\leq 1/K+1 $ which is uniformly bounded.
\medskip
\par  {\em Step 2,} By Step 1 we already know that $f^n:\Gamma\to D_v(c,e^{-\alpha n})$ has uniformly bounded degree. Now we show that the conclusion of the theorem holds.
\medskip
\par Let $ \Gamma' $ be the component of $ f^{-n}(D_v(c,e^{-2\alpha n})) $ contained in $ \Gamma $. By Lemma \ref{Koebe} (2.4) there exist a constant $ \alpha'>0 $ such that $	\meas \;\Gamma'/\meas \;\Gamma\leq e^{-\alpha'n}$. Since $ \infty $ is an attracting fixed point, the set $ E_n $ is uniformly bounded. Thus $ \meas E_n<A $ for some constant $ A>0 $. Thus we have
\begin{equation*}
\frac{\meas \;E'_n}{\meas \;E_n}=\frac{\sum\meas \;\Gamma'}{\sum\meas \;\Gamma}\leq e^{-\alpha'n},
\end{equation*}
where the sum ranges over all possible critical points and connected components.
\medskip
\par  Finally we have shown that 
$ \meas \;E'_n\leq Ae^{-\alpha'n} $. Thus $ \sum_{i=0}^n  \meas \;E'_i<\infty $, and by the Borel-Cantelli lemma for Lebesgue a.e. point in $x\in L_u$ , $ x\notin E'_n $ for large $n$, which means $ \dist _v(f^n(x),\Crit' )\geq e^{-2\alpha n} $. In other words, $x$ slowly approach Crit'. The conclusion follows.
\end{proof}
\medskip

	\section{Positive vertical Lyapunov exponent}
	In this section we prove Theorem 1.4: if $x$ slowly approach Crit' and $\omega(x)\subset J(p)$ then  $\chi_{-}(x)\geq \log \mu_{Exp}$, where $ \mu_{Exp}>1 $ is the constant appearing in the definition of the TCE condition.
\begin{definition}
Let $x\in\Delta\times\mathbb{C}$ satisfies $\dist _v(f^n(x),J(p))\leq \frac{r}{4}$ for every $n\geq 0$, where $r$ is as in the definition of TCE condition. We say a positive integer $n$ is a {\em expanding time} of $x$ if for every $0\leq m\leq n$, the connected component $\Gamma$ of $ f^{-m}(D_v(f^n(x),\frac{r}{4})) $ containing $f^{n-m}(x)$, satisfies $\diam (\Gamma)\leq \mu_{Exp }^{-m}$.
\end{definition}	
\begin{lemma}\label{estimate-exp time}
There exist a uniform constant $\theta>0$ such that if $x\in\Delta\times\mathbb{C}$ satisfies  $\dist _v(f^n(x),J(p))\leq \frac{r}{4}$, then every $n\leq -\theta\log |\pi_1(x)|$ is an expanding time, provided $|\pi_1(x)|$ is small enough.
\end{lemma}
\begin{proof}
Let $n$ be an arbitrary integer, and for $0\leq m\leq n$  let $\Gamma$ be the connected component  of $ f^{-m}(D_v(f^{n}(x),\frac{r}{4})) $ containing $f^{n-m}(x)$. Let $M=\sup_{x\in\Omega} \left( \left| \frac{\partial f}{\partial z}\right| ,\;\left| \frac{\partial f}{\partial  t}\right|\right)$ as before. Then for arbitrary $x_1\neq x_2\in \Gamma$, let $y_1=\pi_2(x_1)$, $y_2=\pi_2(x_2)$. Then for $i=1,2$ we have
\begin{equation*}
\dist _v(p^m(y_i),f^{m}(x_i))\leq M^m|\pi_1(x_i)|\leq M^m|\pi_1(x)|.
\end{equation*}

\par Let $z\in J(p)$ such that $\dist _v(z,f^n(x))\leq \frac{r}{4}$, If $n$ satisfies $M^n|\pi_1(x)|\leq \frac{r}{4}$, then we have
\begin{align*}
\dist (p^m(y_i),z)&\leq \dist _v(p^m(y_i),f^{m}(x_i))+\dist _v(f^m(x_i),z)\\&\leq \dist _v(p^m(y_i),f^{m}(x_i))+\dist _v(f^n(x),f^{m}(x_i))+\dist _v(z,f^n(x))\\&\leq \frac{r}{4}+\frac{r}{4}+\frac{r}{4}<r.\;\;\;\;\;\;\;\;\;\;\;\;\;\;\text{for}\;i=1,2.
\end{align*}
\par Thus $p^m(y_i)$ is in the disk $ D(z,r) $, $i=1,2$. By TCE we know for every connected component $\Gamma'$ of $ f^{-m} D(z,r) $, we have $\diam (\Gamma')\leq \mu_{Exp}^{-m}$, thus $\dist (y_1,y_2)\leq \mu_{Exp}^{-m}$. In other words, $\dist _v(x_1,x_2)\leq \mu_{Exp}^{-m}$. Thus $ \diam (\Gamma)\leq \mu_{Exp}^{-m} $, which implies $n$ is an exponential expanding time of $x$ providing $M^n|\pi_1(x)|\leq \frac{r}{4}$. Thus for any $0< \theta<\frac{1}{\log M} $, the condition $n\leq -\theta\log |\pi_1(x)|$ implies $n$ is an expanding time.
\end{proof}
\medskip

\par The main result of this section is the following.
		\begin{theorem}\label{Positive}
			If $ x\in W^s(J(p)) $ slowly approach Crit', then $ \chi_{-}(x)\geq \log \mu_{Exp} $.
		\end{theorem}
\begin{proof}
Without loss of generality we may assume that 	$x$ satisfies  $\dist _v(f^n(x),J(p))\leq \frac{r}{4}$ for every $n\geq 0$. For fixed sufficiently small $\alpha>0$, there exist $N>0$ such that for $n\geq N$, $ \dist _v(f^n(x),\Crit' )\geq e^{-\alpha n} $ by slow approach. Since $x$ is in $W^s(J(p))$ the orbit of $x$ will stay away from any component of the critical variety which does not belongs to Crit'. Thus we have $ \dist _v(f^n(x),\C )\geq e^{-\alpha n} $ as well.  We may also assume that $|\pi_1(x)|<1$.
\medskip

\par Set $ \delta:=\frac{2\alpha\log M}{\log\mu_{Exp}} $, where $M =\sup_{x\in\Omega} \left( \left| \frac{\partial f}{\partial z}\right| ,\;\left| \frac{\partial f}{\partial t}\right|\right)$. Let $0<\theta'<1$ be a constant that will be determined later. For $(1-\theta')n<s\leq n$, let $ \Gamma_s $ be the connected component of $ f^{s-n}D_v(f^n(x),e^{-\delta n}) $ containing $f^s(x)$. Set $ \delta_0=\frac{\delta}{2\log M} $, and note that for all large $n$, $\frac{r}{4}M^{- \delta_0 n }\geq e^{-\delta n}$. In particular
\begin{equation*}
f^{ \delta_0 n }(D_v(f^n(x),e^{-\delta n}))\subset D_v(f^{n+ \delta_0 n }(x),r/4), 
\end{equation*} hence
\begin{equation*}
D_v(f^n(x),e^{-\delta n})\subset \Comp _{f^n(x)}f^{- \delta_0 n }(D_v(f^{n+ \delta_0 n }(x),\frac{r}{4})).
\end{equation*}  Here $ \Comp _y $ denotes the connected component containing $y$.
\medskip
\par We claim that there exist $0<\theta'<1$ such that for large $n$, $\theta'n+ \delta_0 n $ is an expanding time of $f^{n-\theta'n} (x)$. Indeed by Lemma \ref{estimate-exp time}, for every $m\leq n$ large, $ -m\theta\log|\lambda| $ is an expanding time of $f^m(x)$. Provided $\alpha$ is sufficiently small we have $\delta_0$ is sufficiently small as well, thus $\theta'=\frac{\theta\log|\lambda|+\delta_0}{\theta\log|\lambda|-1}$ is a positive number.  We conclude that $\theta'n+ \delta_0 n $ is an expanding time of $f^{n-\theta'n} (x)$.
 Thus for every $n-\theta'n<s\leq n$, we have
\begin{align*}
\diamComp_{f^s(x)}f^{s-n}(D_v(f^n(x),e^{-\delta n}))&\leq \diamComp_{f^s(x)}f^{s-(n+ \delta_0 n )}(x)(D_v(f^{n+ \delta_0 n }(x),\frac{r}{4}))\\&\leq \mu_{Exp}^{s-(n+ \delta_0 n )}\leq \mu_{Exp}^{- \delta_0 n }=e^{-\alpha n}.\;(\text{By the choice of}  \;\delta_0.)
\end{align*}
 Thus from $ \dist _v(f^s(x),\C )\geq e^{-\alpha s} $ we get that
\begin{equation*}
\Comp_{f^s(x)}f^{s-n}(D_v(f^n(x),e^{-\delta n}))\cap\C =\emptyset.
\end{equation*}
 This implies that $ f^{\theta' n} $ restricted to $ \Comp_{f^{n-\theta'n+1}(x)}f^{-\theta'n+1}(D_v(f^n(x),e^{-\delta n})) $ is univalent. Since  $\theta'n$ is an expanding time of $f^{n-\theta'n} (x)$, we have
 \begin{equation}\label{5.1}
 \diamComp_{f^{n-\theta'n+1}(x)}f^{-\theta'n+1}(D_v(f^n(x),e^{-\delta n}))\leq \mu_{Exp}^{-\theta'n}.
 \end{equation}Thus by (\ref{5.1}) and Koebe distortion, there is a uniform constant $C>0$ such that
\begin{equation*}
\left| \left( Df^{\theta'n} \right) _{f^{n-\theta'n+1}(x)}(v)\right| \geq Ce^{-\delta n}\mu_{Exp}^{\theta'n},
\end{equation*}
where $v$ is the unit vertical vector.
\medskip
\par Next we replace $ f^n(x) $ by $ f^{n-\theta'n}(x) $, and repeat the argument above, we get an estimate
\begin{equation*}
\left|\left(  Df^{\theta'n_1}\right)  _{f^{n_1-\theta'n_1+1}(x)}(v)\right| \geq Ce^{-\delta n_1}\mu_{Exp}^{\theta'n_1},
\end{equation*}
where $n_1=n-\theta'n$.
\medskip
\par We define $ n_m:=n_{m-1}-\theta'n_{m-1} $, $m\geq 1$, we set $n_0=n$. We can repeat this procedure until for some $k$, $n-\sum_{i=0}^k n_i\theta'\leq N$. In this final time we can not define $n_k$ as $ n_k=n_{k-1}-\theta'n_{k-1} $, instead we choose the final $ n_k $  to satisfying $ N<n-\sum_{i=0}^k n_i\theta'\leq 2N$. Combining these estimates, take the product of this derivatives, we have
\begin{align*}
\left| \left( Df^n\right) _{x}(v)\right| &\geq \varepsilon_1 \left| \left( Df^{\theta'n_k} \right) _{f^{n_k-\theta'n_k+1}(x)}(v)\right|\cdots \left| \left( Df^{\theta'n} \right) _{f^{n-\theta'n+1}(x)}(v)\right| \\&\geq \varepsilon_1 Ce^{-\delta n_k}\mu_{Exp}^{\theta'n_k}\cdots Ce^{-\delta n}\mu_{Exp}^{\theta'n}\\&\geq \varepsilon_1 C^{k+1}e^{-\delta(n-2N)/\theta'}\mu_{Exp}^{n-2N},
\end{align*}
where we take $ \varepsilon_1=\min_{0\leq j\leq 2N }|Df^j|_{x}(v)| $.
\medskip
\par It is not hard to give an upper bound of $k$. Indeed we let $S_m:=n-\sum_{i=0}^m n_i\theta'$, then $ S_m $ satisfies $S_m=(1-\theta')S_{m-1}$, for $1\leq m\leq k$. Thus we get $S_k=(1-\theta')^{k+1}n$. Now $S_k>N$ implies $k<\frac{\log N- \log n}{\log (1-\theta')}-1$. Thus $ C^{k+1} $ is a sub-exponentially large term with respect to $n$.
\medskip
\par Taking the limit in the above inequality we get 
\begin{equation*}
\chi_{-}(x)=\liminf _{n\to\infty}\frac{1}{n} \log|Df^n|_{x}(v)|\geq \log\mu_{Exp}-\frac{\delta}{\theta'}.
\end{equation*}
Letting $ \alpha\to 0 $ then  $\delta/\theta'\to 0$ as well, and we get $\chi_{-}(x)\geq \log \mu_{Exp}$. 
\end{proof}
\medskip
\begin{corollary}
There are no wandering Fatou components in $\Delta\times\mathbb{C}$, $W^s(J(p))=J(f)$, and Fatou set $F(f)$ is equal to the union of basins of attracting cycles. Moreover  Lebesgue a.e. point $x\in J(f)$ slowly approach Crit' and $\chi_{-}(x)\geq \log \mu_{Exp}$.
\end{corollary}
\begin{proof}
By Theorem 4.4 and Theorem 5.3, for Lebesgue a.e. point $x\in W^s(J(p))$, $x$ slowly approach Crit' and $\chi_{-}(x)\geq \log \mu_{Exp}$. We also know $W^s(J(p))$ is the union of $J(f)$ and the wandering Fatou components.

\par It is clear that points in the Fatou set can not have a positive vertical Lyapunov exponent, thus there are no wandering Fatou component in $\Delta\times\mathbb{C}$, and $W^s(J(p))=J(f)$. Since every attracting basin of $p$ bulges to an attracting basin of $f$, for every point $x$ such that $x\notin W^s(J(p))$ we get that $x$ is in a basin of attracting cycle. Thus the Fatou set is the union of  basins of attracting cycles.
\end{proof}
\medskip
	\section{The Julia set $J(f)$ has Lebesgue measure zero}
	In this section we prove Theorem 1.5, thus finishing the proof of the main theorem. In the following we assume that $x\in J(f)$ is both slowly approaching Crit' and satisfies $\chi_{-}(x)\geq \log \mu_{Exp}$. We begin with a definition.
	\begin{definition}
		Let $1<\sigma<\log\mu_{Exp}$, Let $m$ be a positive integer. We say $m$ is a {\em $\sigma$-hyperbolic time} for $x$ if 
		\begin{equation*}
	\left| \left( Df^{m-i}\right) _{f^i(x)}(v)\right| \geq \sigma^{m-i}
		\end{equation*}
		holds for each $0\leq i\leq m-1$, and $v$ is the unit vertical vector.
	\end{definition}
We fix once for all $1<\sigma<\sigma'<\mu_{Exp}$.
	\medskip
	\par Since $\chi_{-}(x)\geq\log \mu_{Exp}$, the hyperbolic times have positive density by Pliss's Lemma \cite{pliss1972conjecture}, in the following sense:
	\begin{lemma}\label{Pliss}
		There is a constant $\theta>0$ such that if we consider the set 
		\begin{equation*}
		H_n=\left\lbrace m\in \left\lbrace1,\cdots, n\right\rbrace: m\;\text{is a}\;\sigma'\text{-}\text{hyperbolic time for}\; x \right\rbrace ,
		\end{equation*}
		then for large $n$ we have
		\begin{equation*}
		\frac{\# H_n}{n}>\theta.
		\end{equation*}
	\end{lemma}
For the proof see \cite[Theorem 3.1]{levin2016lyapunov}.
\medskip
\par Next for a positive integer $n$ we define $\phi(f^n(x)):=-\log\dist _v(f^n(x),\Crit' )$. Multiplying the metric by a constant we can further assume $\phi$ is a positive function. We show that
\begin{lemma}\label{2-DPU}
There exists a constant $C=C(x)>0$ such that for every $n\geq 0$,
\begin{equation*}
\sum_{k=0}^{n-1}\phi(f^k(x))\leq Cn.
\end{equation*}
\end{lemma}
\begin{proof}
Fix $\alpha>0$ small, by slow approach for large $n$ we have $\dist _v(f^n(x),\Crit' )\geq e^{-\alpha n}$, or in other words $\phi(f^n(x))\leq \alpha n$.
\medskip
\par We claim that there exist constant $0<\theta<1$  and $C_1>0$ such that such for large $n$ we have $\sum_{k=\theta n}^{n-1}\phi(f^k(x))\leq C_1n. $ To show this, let $ z=\pi_2(f^{\theta n}(x)) $. By Lemma \ref{DPU} we have
\begin{equation*}
\sum_{\begin{subarray}{c}j=0\\\text{except}\;M\;\text{terms}\end{subarray}}^{(1-\theta)n-1}\phi(p^j(z))\leq Q(1-\theta)n.
\end{equation*}
\par In particular $\phi(p^j(z))\leq Q(1-\theta)n$ holds for $j$ appearing in above sum. On the other hand there is a constant $ K>0 $ so that $\dist _v(p^j(z),f^{\theta n+j}(x))\leq K^{j}|\lambda|^{\theta n}$. We choose $\theta$ sufficiently close to $1$ so that $e^{Q(\theta-1)n}-K^{(1-\theta) n}|\lambda|^{\theta n}\geq e^{2Q(\theta-1)n}$. Thus we have
\begin{align*}
\dist _v(f^{\theta n+j}(x),\Crit' )&\geq \dist (p^j(z),\Crit' )-\dist _v(f^{\theta n+j}(x),p^j(z) )\\&\geq e^{-\phi(p^j(z))}-K^{(1-\theta) n}|\lambda|^{\theta n}\\&\geq e^{-\phi(p^j(z))}+e^{2Q(\theta-1)n}-e^{Q(\theta-1)n}\geq e^{-2\phi(p^j(z))},
\end{align*}
which implies $\phi(f^{\theta n+j}(x))\leq 2\phi(p^j(z))$. Then we get
\begin{equation*}
\sum_{\begin{subarray}{c}k=\theta n\\\text{except}\;M\;\text{terms}\end{subarray}}^{n-1}\phi(f^k(x))\leq 2Q(1-\theta)n.
\end{equation*}
Together with slow approach $\phi(f^n(x))\leq \alpha n$ we have
\begin{equation*}
\sum_{k=\theta n}^{n-1}\phi(f^k(x))\leq \left( 2Q(1-\theta)+M\alpha\right)n.
\end{equation*} 
Setting $ C_1:=\left( 2Q(1-\theta)+M\alpha\right) $ we get the conclusion.
\medskip
\par Repeat the above argument we get the estimate in the time $\theta^2 n$ to $\theta n$ we get 
\begin{equation*}
\sum_{k=\theta^2 n}^{\theta n-1}\phi(f^k(x))\leq C_1\theta n.
\end{equation*}
Keep repeating the above argument in the time $\theta^{j} n$ to $\theta^{ j-1}n$ until for some $j$, the slow approach property $\dist _v(f^{\theta^{j+1} n}(x),\Crit' )\geq e^{\theta^{j+1} n}$ does not holds. The final step is a bounded time, and the sum of $\phi(f^k(x))$ in this bounded time is bounded by a constant depending on $x$. Summing up there is a constant $C=C(x)>0$ such that
\begin{equation*}
\sum_{k=0}^{n-1}\phi(f^k(x))\leq Cn.
\end{equation*}
\end{proof}
\par We now introduce some notions from \cite[Theorem 3.1]{levin2016lyapunov}. Given $K>0$ we define the shadow $S(j,K)$ of a positive integer $j$ to be the following interval of the real line:
\begin{equation*}
S(j,K):=(j,j+K\phi(f^j(x))].
\end{equation*}
For a positive integer $N$, let $A(N,K)$ be the set of all positive integer $n$ such that at most $N$ integers $j$ satisfy $n\in S(j,K)$. The following lemma and Theorem are proved in \cite[Theorem 3.1]{levin2016lyapunov}, and both rely on the one-dimensional DPU lemma (Lemma \ref{DPU}). In our case we can replace the DPU lemma by Lemma \ref{2-DPU}, and get exactly the same statements.
\begin{lemma}\label{density}
	For any $N$ and $K$, for $n$ sufficiently large we have
	\begin{equation*}
	\frac{\left\lbrace A(N,K)\cap\left\lbrace 1,\cdots,n \right\rbrace \right\rbrace} {n}\geq 1-\frac{CK}{N+1}.
	\end{equation*}
\end{lemma}
\medskip
\begin{theorem}\label{good time}
Suppose $m$ is an $ \sigma' $-hyperbolic time and $m\in A(N,\frac{1}{\log \sigma})$, then there exist a constant $\delta>0$ such that if we let $V_m$ be the connected component of $f^{-m}D_v(f^m(x),\delta)$ containing $x$, then $f^m:V_m\to D_v(f^m(x),\delta)$ has degree at most $N$.
\end{theorem}
\medskip
If we choose $N$ sufficiently large, by  Lemma \ref{density} the density of $A(N,\frac{1}{\log \sigma})$ is close to 1. Together with Lemma \ref{Pliss}, we get for $m$ large, the intersection $H'_m:=H_m\cap A(N,\frac{1}{\log \sigma})$ has uniform positive density when $m\to\infty$. Now by Theorem \ref{good time} we have
\begin{corollary}\label{distortion}
For large $n$ there exist a subset $ H'_n\subset \left\lbrace 1,\cdots,n\right\rbrace  $ and constants $\alpha>0$, $\delta>0$ such that $\# H'_n\geq \alpha n$ and for every $m\in  H'_n$,  $f^m:V_m\to D_v(f^m(x),\delta)$ has degree at most $N$.
\end{corollary}
\medskip
\par Now we are able to establish the main result of this section.
\begin{theorem}
The Julia set $ J(f) $ in the basin of $L$ has Lebesgue measure zero.
\end{theorem}
\begin{proof}
Let $\delta$ and $N$ be as in Corollary \ref{distortion}. First we observe that the Fatou set $ F(p) $ in the invariant fiber $L$ has full Lebesgue measure, as a consequence of TCE. Let $\Omega$ be a relatively compact subset of $F(p)$, then there exist a constant $\varepsilon>0$ such that for $x$ satisfying $|\pi_1(x)|<\epsilon$, $\pi_2(x)\in\Omega$, we have $x\in F(f)$. Thus for $y\in\Delta\times\mathbb{C}$ with $\pi_1(y)$ sufficiently small we have
\begin{equation*}
\frac{\meas D_v(y,\delta/2)\cap J(f)}{\meas D_v(y,\delta/2)}\leq \varepsilon,
\end{equation*}
here meas denote the one-dimensional Lebesgue measure, and $\varepsilon$ is a constant to be determined in the next paragraph.
\medskip
\par Now we argue by contradiction. Suppose $J(f)$ has positive Lebesgue measure, then by the Lebesgue density theorem and the Fubini theorem there exist $x\in J(f)$ such that $x$ is a Lebesgue density point in the vertical line containing $x$. We may also assume that $x$ has positive Lyapunov exponent and slowly approach Crit'. By Corollary \ref{distortion} there is a sequence of positive integers $ \left\lbrace n_0,\cdots,n_k,\cdots\right\rbrace  $ such that $f^{n_k}:V_{n_k}\to D_v(f^{n_k}(x),\delta)$ has degree bounded by $N$ for all $k\geq 0$. Let $V'_{n_k}$ be the connected component of $ f^{-n_k}D_v(f^{n_k}(x),\delta/2) $ containing $x$. Let $\varepsilon$ sufficiently small such that $C_4 \varepsilon^{2^{-N}}<<1$, where $C_4$ is the constant in Lemma \ref{Koebe} (2.4). By Lemma \ref{Koebe} (2.4) we have
\begin{equation*}
\frac{\meas V'_{n_k}\cap J(f)}{\meas V'_{n_k}}\leq C,
\end{equation*}
where $C<1$ does not depend on $k$. Again by Lemma \ref{Koebe} (2.1) $\diam V'_{n_k}$ is exponentially small, and also by Lemma 2.4 (2.2), $V'_{n_k}$ has uniformly good shape (the ratio of the diameter and the inradius of $V'_{n_k}$ is uniformly bounded). This contradicts that $x$ is a Lebesgue density point. Thus $J(f)$ must have Lebesgue measure zero.
\end{proof}
\medskip
\begin{remark}
The main theorem also holds in a slightly more general setting. Let $\Delta$ be a disk. Let $f:\Delta\times \mathbb{P}^1\to \Delta\times \mathbb{P}^1$ be a skew product holomorphic map in the following form: $f(t,z)=\left( \lambda t, h(t,z)\right) $, where $|\lambda|<1$ and $h(t,z)$ is a rational map in $z$ for fixed $t$. We assume moreover that the degree of $h(t,z)$ in $z$ is a constant for $t\in \Delta$. Let $L=\left\lbrace t=0\right\rbrace $ be the invariant fiber and let $ p=f|_L $. Assume $p$ has non-empty Fatou set and $p$ satisfies either 1.TCE+WR or 2.Positive Lyapunov. Then the Fatou set of $f$ is the union of the basins of attracting cycles and the Julia set of $f$ has Lebesgue measure zero.
\medskip
\par We notice that $\Delta\times \mathbb{P}^1$ can not be embedded into $\mathbb{P}^2$ since any two projective lines in $\mathbb{P}^2$ have non-trivial intersection. However $\Delta\times \mathbb{P}^1$ can be embedded into $\mathbb{P}^1\times \mathbb{P}^1$, and the $f$ above can be realized as a semi-local restriction of a globally defined meromorphic map from $\mathbb{P}^1\times \mathbb{P}^1$ to $\mathbb{P}^1\times \mathbb{P}^1$. To construct such examples, we start with a skew product meromorphic self map $f:\mathbb{P}^1\times \mathbb{P}^1\to \mathbb{P}^1\times \mathbb{P}^1$, $f(t,z)=(g(t),h(t,z))$, where $g$ is a one-variable rational function and $h$ is a two-variable rational function. The function $h$ has finite number of indeterminacy points, and $f$ is holomorphic outside these indeterminacy points. We choose $g$ such that $g$ has an attracting fixed point $t_0$, and there are no indeterminacy points in the line $\left\lbrace t_0\right\rbrace \times \mathbb{P}^1 $. Thus there exist a small neighborhood $\Omega$ of $\left\lbrace t_0\right\rbrace \times \mathbb{P}^1 $ such that $f:\Omega\to\Omega$ is holomorphic, thus $f$ is a skew product holomorphic map. We note that indeterminacy points are necessary since a globally holomorphic self map of $\mathbb{P}^1\times \mathbb{P}^1$ must be a product map, see \cite[Remark 1.6]{favre2001dynamique}.
\end{remark}
\medskip
\begin{appendix}
\section{Relations between non-uniformly hyperbolic conditions}
In this Appendix we study the relations between non-uniform hyperbolic conditions given in section 2. In the following we assume $f$ is a rational map on $\mathbb{P}^1$ and distance are relative to the spherical metric.
\begin{definition}
	A rational map $f$ satisfies {\em Slow Recurrence condition with exponent $\alpha$} ({\em SR($ \alpha $)} for short) if for every critical point $c\in J(f)$, there exist an $\alpha>0$ such that
\begin{equation*}
\dist (f^n(c),\Crit' )\geq e^{-n\alpha}\;\text{for n large}.
\end{equation*}
\end{definition}
\medskip
\begin{lemma}\label{SR}
WR($ \eta $,$ \iota $) implies SR($ \alpha $) for some $ \alpha(\iota)>0 $, and $\alpha\to 0$ when $\iota\to 0$.
\end{lemma}
\begin{proof}
By the definition of WR($ \eta $,$ \iota $) in particular we have $-\log|f'(f^n(c))|<n\iota+C_0$ for every $n\geq 0$, which is equivalent to say $ |f'(f^n(c))|>e^{-n\iota-C_0} $, then it is straightforward that there exist an $\alpha>0$ such that $\dist (f^n(c),\Crit' )\geq e^{-n\alpha}$ for large $n$, and $\alpha\to 0$ when $\iota\to 0$.
\end{proof}

\medskip
\begin{lemma}
Positive Lyapunov implies SR($ \alpha $) for every $\alpha>0$.
\end{lemma}
\begin{proof}
By the definition of Positive Lyapunov in particular  we have 
\begin{equation*}
\lim_{n\to 0}\frac{\log |f'(f^n(c))|}{n}= 0.
\end{equation*}
Thus for every $\beta>0$ we have $ |f'(f^n(c))|>e^{-\beta n} $ for large $n$ . Similarly to Lemma \ref{SR}, there exist an $\alpha>0$ such that $\dist (f^n(c),\Crit' )\geq e^{-n\alpha}$ for large $n$, and $\alpha\to 0$ when $\beta\to 0$. Thus $f$ satisfies SR($\alpha$) for every $\alpha>0$.
\end{proof}
\medskip
\begin{lemma}
TCE+WR($ \eta $,$ \iota $) with $\eta$ small or TCE+SR($ \alpha $) with $\alpha$ small implies CE.
\end{lemma}
\begin{proof}
  This lemma was proved by Li in \cite{li2017topological} for real maps, and Li's argument can also apply to rational maps. Here we give a simple proof for rational maps. This kind of argument has already appeared in \cite{przytycki2003equivalence}. By Lemma \ref{SR} it is sufficient to prove TCE+SR($ \alpha $) with $\alpha$ small implies CE. Let $v=f(c)$. Let $M=\sup_{x\in\mathbb{P}^1}|f'(x)|$. Set $\alpha_1=\frac{2\alpha\log M}{\log\mu_{Exp }} $ and $ \varepsilon=\frac{\alpha_1}{2\log M} $. Note that for all large $n$, $ rM^{-\varepsilon n}\geq e^{-\alpha_1 n}$, here $r$ is the constant appearing in the definition of the TCE condition. We note that 
   \begin{equation*}
     D(f^n(v),e^{-\alpha_1 n})\subset \Comp _{f^n(v)}f^{-\varepsilon n}(D(f^{n+\varepsilon n}(v),r)), 
   \end{equation*} here $ \Comp _y $ means the connected component containing $y$ 
\medskip
\par   For every $0\leq s\leq n$ and $ n $ large we have
\begin{equation*}
\diamComp_{f^s(v)}f^{s-(n+\varepsilon n)}(D(f^{n+\varepsilon n}(v),r))\leq \mu_{Exp }^{s-(n+\varepsilon n)}\leq \mu_{Exp }^{-\varepsilon n}.
\end{equation*}
Since $ D(f^n(v),e^{-\alpha_1 n})\subset \Comp _{f^n(v)}f^{-\varepsilon n}(D(f^{n+\varepsilon n}(v),r)) $, we have
\begin{equation*}
\diamComp_{f^s(v)}f^{s-n}(D(f^{n}(v),e^{-\alpha_1 n}))\leq  \mu_{Exp }^{-\varepsilon n}= e^{-\alpha n}
\end{equation*}

By SR($ \alpha $), for all large $n$ and all $0\leq s\leq n$ we have
\begin{equation*}
\Comp _{f^s(v)}f^{s-n}(D(f^{n}(v),e^{-\alpha_1 n}))\cap\C =\emptyset.
\end{equation*}
Hence $ f^n $ restricted to $ \Comp _{v}f^{-n}(D(f^{n}(v),e^{-\alpha_1 n})) $ is univalent, by Koebe distortion lemma there exist a constant $ C>0 $ such that $ |(f^n)'(v)|\geq C e^{-\alpha_1 n}/\mu_{Exp }^{-n} $. Since $\alpha$ is small we get $f$ is CE.
\end{proof}

\section{Genericity of non-uniformly hyperbolic conditions}
In this Appendix we give some families of polynomials satisfying the consitions in our main theorem (i.e. TCE+WR or Positive Lyapunov).
\medskip
\par  In real dynamics, the WR condition was first introduced in the Tsujii's paper \cite{tsujii1993positive}. Avila and Moreira proved that CE+WR condition is generic (has full Lebesgue measure) in every non-trivial analytic family of  S-unimodal maps \cite{avila2001statistical}. The condition CE+WR was also studied by Luzzatto and Wang \cite{luzzatto2006topological}, and also by Li \cite{li2017topological} in relation to topological invariance.  For the Positive Lyapunov condition, Avila and Moreira proved that this condition  is generic (has full Lebesgue measure) in every non-trivial analytic family of quasi-quadratic maps \cite{avila2005statistical}. The quadratic family $ \left\lbrace f_t(x)=t-x^2 \right\rbrace  $ for $\frac{-1}{4}\leq t\leq 2$ is obviously a non-trivial analytic family of quasi-quadratic maps, and of S-unimodal maps. So our theorem also applies for these real polynomials (seen as complex dynamical systems).
\medskip
\par In the rational map case it was shown by Astorg, Gauthier, Mihalache and Vigny that the CE and WR($\eta,\iota$) with arbitrarily small $\iota$ are robust \cite[Lemma 5.5]{astorg2017collet} (in the sense that there is a positive Lebesgue measure set in the parameter space satisfying both these two conditions).
\medskip
\par  Next we consider the family of uni-critical polynomials, i.e.  the family $ \left\lbrace f_c(z)=z^d+c \right\rbrace $, $c\in\mathbb{C}$ and $d\geq 2$ an integer. We let  $ \mathcal{M}_d $ be the connectedness locus, and let $\partial \mathcal{M}_d$ be the bifurcation locus. There is a harmonic measure (with pole at $\infty$) supported on $\partial \mathcal{M}_d$. It is shown by Graczyk and Swiatek in \cite{graczyk2015lyapunov} that for a.e. $ c\in\partial \mathcal{M}_d $ in the sense of harmonic measure the Lyapunov exponent at $c$ exist and is equal to $\log d$. 
\medskip
\par For the WR condition, the author is told by Jacek Graczyk that WR condition is actually generic in the sense of harmonic measure in the family of uni-critical polynomials. We thank Jacek Graczyk for kindly let us write down his argument here.

\begin{theorem}
 In the uni-critical family $\left\lbrace f_c(z)=z^d+c \right\rbrace $, $d\geq 2$, a.e. $x\in \partial \mathcal{M}_d$ in the sense of harmonic measure satisfies WR condition.
\end{theorem}
\begin{proof}
For $c\in\mathbb{C}$, let $\omega_c$ be the unique measure of maximal entropy of $f_c$. It is a result of Brolin \cite{brolin1965invariant} that its Lyapunov exponent $\int \log|f'_c(z)|
\;d\omega_c$ is is equal to $\log d$. Since $|f_c'(z)|=d|z^{d-1}|$ we get 
$\int -\log |z|\;d\omega_c=0$.
\par We define the truncation function $H_\delta$ on $J(f_c)$ for $\delta>0$ as 
$$ H_\delta(z)=\left\{
\begin{aligned}
&-\log |z|,\;\text{when} \;|z|>\delta, \\&
-\log |\delta|,\;\text{when}\; |z|\leq\delta .
\end{aligned}
\right.
$$
\par Thus $H_\delta$ is a continuous function, and $H_\delta\to -\log|\cdot|$ when $\delta\to 0$ in $L^1(\omega_c)$. According to \cite{graczyk2015lyapunov} section 1.1, for a.e. $c\in\partial\mathcal{M}_d$ in the sense of harmonic measure, the critical value $c$ of $f_c$ is typical with respect to $\omega_c$. Here typical means for every continuous function $H$ on $J(f_c)$,
\begin{equation}
\lim_{n\to\infty}\frac{1}{n} \sum_{i=0}^{n-1}H(f_c^i(c))=\int H \;d\omega_c.
\end{equation}
\par Applying (B.1) to $H_\delta$, together with that fact that $H_\delta\to -\log|\cdot|$ in  $L^1(\omega_c)$ as $\delta\to 0$  we get 
\begin{equation}
\lim_{\delta\to 0}\lim_{n\to\infty}\frac{1}{n} \sum_{i=0}^{n-1}H_\delta(f_c^i(c))=\int -\log |z|\;d\omega_c=0.
\end{equation}
\par On the other hand, for every $\delta>0$ let $ F_\delta $ be a positive continuous function such that $\supp F_\delta\subset D(0,2\delta)$, $\left\| F_\delta\right\|_{\infty}=1$  and $F_\delta\geq \chi_{D(0,\delta)}$. Then for for a.e. $c\in\partial\mathcal{M}_d$ in the sense of harmonic measure we have 
\begin{equation}
 \lim_{n\to\infty}\frac{-\log\delta}{n} \sum_{i=0}^{n-1} F_\delta(f_c^i(c))=-\log\delta \int F_\delta \;d\omega_c\leq -\log\delta \;\omega_c(D(0,2\delta)).
\end{equation}
By \cite[ Lemma 4]{przytycki1990harmonic} (or by the fact that the dynamical Green function is H\"older continuous, see \cite{sibony1999dynamique} Theorem 1.7.3), for every $c\in \mathbb{C}$ there exist constants $C=C(c)>0$, $\alpha=\alpha(c)>0$ such that for every $r>0$ we have $\omega_c(D(0,r))\leq Cr^\alpha$.

\medskip
 \par Thus for $c$ satisfying (B.3) we have
\begin{equation*}
\lim_{n\to\infty}\frac{-\log\delta}{n} \sum_{i=0}^{n-1} F_\delta(f_c^i(c))\leq -\log \delta C (2\delta)^\alpha.
\end{equation*}
Thus we have
\begin{align*}
\lim_{\delta\to 0}\limsup_{n\to\infty}\frac{-\log\delta}{n} \sum_{i=0}^{n-1} \chi_{D(0,\delta)}(f_c^i(c))&\leq \lim_{\delta\to 0}\limsup_{n\to\infty}\frac{-\log\delta}{n} \sum_{i=0}^{n-1} F_\delta(f_c^i(c))\\&\leq \lim_{\delta\to 0}-\log \delta C (2\delta)^\alpha=0.
\end{align*}
We conclude that 
\begin{equation}
\lim_{\delta\to 0}\limsup_{n\to\infty}\frac{-\log\delta}{n} \sum_{i=0}^{n-1} \chi_{D(0,\delta)}(f_c^i(c))=0.
\end{equation}
It is easy to check
\begin{equation*}
\sum_{\begin{subarray}{c}i=0\\f_c^i(c)\notin D(0,\delta)\end{subarray}}^{n-1}-\log|f_c^i(c)|=\sum_{i=0}^{n-1}H_\delta(f_c^i(c))+\log\delta\sum_{i=0}^{n-1} \chi_{D(0,\delta)}(f_c^i(c))
\end{equation*}
Combining (B.2) and (B.4) we get for a.e. $c\in\partial\mathcal{M}_d$ in the sense of harmonic measure
\begin{equation}
\lim_{\delta\to 0}\liminf_{n\to\infty}\frac{1}{n}\sum_{\begin{subarray}{c}i=0\\f_c^i(c)\notin D(0,\delta)\end{subarray}}^{n-1}-\log|f_c^i(c)|=0.
\end{equation}
Finally by the main theorem of \cite{graczyk2015lyapunov}, for a.e. $c\in\mathcal{M}_d$ in the sense of harmonic measure
\begin{equation*}
\lim_{n\to\infty}\frac{1}{n}\log|(f_c^n)'(c)|=\log d,
\end{equation*}
which is equivalent to (since $|f_c'(z)|=d|z^{d-1}|$ )
\begin{equation}
\lim_{n\to\infty}\frac{1}{n}\sum_{i=0}^n -\log|f^i_c(c)|=0.
\end{equation}
Combining (B.5) and (B.6) we get for a.e. $c\in\mathcal{M}_d$ in the sense of harmonic measure
\begin{equation}
\lim_{\delta\to 0}\limsup_{n\to\infty}\frac{1}{n}\sum_{\begin{subarray}{c}i=0\\d(f^i_c(c),0)\leq \delta\end{subarray}}^{n-1}-\log|f_c^i(c)|=0.
\end{equation}
By Przytycki's lemma (Lemma \ref{Przytycki}) we have 
\begin{equation}
\lim_{\delta\to 0}\lim_{n\to\infty}\frac{1}{n}\sum_{i=0}^{n-1} \chi_{D(0,\delta)}(f_c^i(c))=0.
\end{equation}
Since $|f_c'(f_c^i(c))|=d|f_c^i(c)|^{d-1}$, combining (B.7) and (B.8) we get
\begin{equation*}
\lim_{\delta\to 0}\limsup_{n\to\infty}\frac{1}{n}\sum_{\begin{subarray}{c}i=0\\d(f^i_c(c),0)\leq \delta\end{subarray}}^{n-1}-\log|f_c'(f_c^i(c))|=0.
\end{equation*}
Thus the proof is complete.
\end{proof}

\end{appendix}

\end{document}